\documentclass[12pt]{article}
\author{Christian d'Elb\'{e}e, Yatir Halevi, and Will Johnson}
\title{The classification of dp-minimal integral domains}

\usepackage{amsmath, amssymb, amsthm}    	
\usepackage{fullpage} 	
\usepackage{amscd}
\usepackage{hyperref}
\usepackage[all]{xy}
\usepackage{centernot}
\usepackage{comment}
\usepackage{xcolor}
\usepackage{enumitem}
\usepackage{stmaryrd}

\usepackage{todonotes}

\def\Ind#1#2{#1\setbox0=\hbox{$#1x$}\kern\wd0\hbox to 0pt{\hss$#1\mid$\hss}
	\lower.9\ht0\hbox to 0pt{\hss$#1\smile$\hss}\kern\wd0}

\def\notind#1#2{#1\setbox0=\hbox{$#1x$}\kern\wd0
	\hbox to 0pt{\mathchardef\nn=12854\hss$#1\nn$\kern1.4\wd0\hss}
	\hbox to 0pt{\hss$#1\mid$\hss}\lower.9\ht0 \hbox to 0pt{\hss$#1\smile$\hss}\kern\wd0}

\newcommand{\can}{\mathrm{can}}

\newcommand{\alg}{\mathrm{alg}}

\newcommand{\ACF}{\mathrm{ACF}}

\newcommand{\RCF}{\mathrm{RCF}}
\newcommand{\ACVF}{\operatorname{ACVF}}

\newcommand{\Frac}{\operatorname{Frac}}

\newcommand{\characteristic}{\operatorname{char}}

\newcommand{\res}{\operatorname{res}}
\newcommand{\rv}{\operatorname{rv}}
\newcommand{\Aut}{\operatorname{Aut}}

\newcommand{\tp}{\operatorname{tp}}

\newcommand{\dpr}{\operatorname{dp-rk}}

\def\seq{\subseteq}

\newcommand{\set}[1]{\left\{ {#1} \right\}}

\newtheorem{theorem}{Theorem}[section]
\newtheorem{lemma}[theorem]{Lemma}
\newtheorem{theorem-not}[theorem]{Theorem (or not)}
\newtheorem{corollary}[theorem]{Corollary}
\newtheorem{fact}[theorem]{Fact}
\newtheorem{fact-q}[theorem]{Fact(?)}
\newtheorem{proposition}[theorem]{Proposition}
\newtheorem*{theorem-star}{Theorem}
\newtheorem*{proposition-star}{Proposition}
\newtheorem*{lemma-star}{Theorem}
\newtheorem*{conjecture-star}{Conjecture}
\newtheorem*{fact-star}{Fact}

\theoremstyle{definition}
\newtheorem{definition}[theorem]{Definition}
\newtheorem{informal-definition}[theorem]{Informal Definition}
\newtheorem{theorem-definition}[theorem]{Theorem-Definition}

\newtheorem{example}[theorem]{Example}

\newtheorem{remark}[theorem]{Remark}

\newtheorem{claim}[theorem]{Claim} 

\newtheorem*{observation-star}{Observation}
\newtheorem*{acknowledgment}{Acknowledgments}

\newcommand{\Qq}{\mathbb{Q}}

\newcommand{\Rr}{\mathbb{R}}
\newcommand{\Cc}{\mathbb{C}}

\newcommand{\Zz}{\mathbb{Z}}

\newcommand{\Nn}{\mathbb{N}}

\newcommand{\Mm}{\mathbb{M}}
\newcommand{\Ff}{\mathbb{F}}

\newcommand{\Oo}{\mathcal{O}}
\newcommand{\mm}{\mathfrak{m}}
\newcommand{\pp}{\mathfrak{p}}
\newcommand{\qq}{\mathfrak{q}}

\newenvironment{claimproof}[1][\proofname]
{
                 \proof[#1]
                 
}
{
                 \endproof
}

\let\phi\varphi

\def\yatir{\textcolor{magenta}}

\begin{document}

\maketitle

\begin{abstract}
  We classify dp-minimal integral domains, building off the existing
  classification of dp-minimal fields and dp-minimal valuation rings.
  We show that if $R$ is a dp-minimal integral domain, then $R$ is a
  field or a valuation ring or arises from the following
  construction: there is a dp-minimal valuation overring $\Oo \supseteq R$, a proper ideal $I$ in $\Oo$, and a finite subring $R_0 \subseteq \Oo/I$ such that $R$
  is the preimage of $R_0$ in $\Oo$. 
\end{abstract}

\section{Introduction}
A dp-minimal domain is an integral domain whose first order theory is of dp-rank $1$; see Section \ref{S:pre} for the definition.  In \cite{dp-min-domains}, the first two authors studied under which algebraic conditions dp-minimal domains are in fact valuation rings. After noting that any such ring $R$ must be local, the answer was that $R$ is a valuation ring if and only if its residue field is infinite or its residue field is finite and its maximal ideal is principal. The third author gave a complete classification of dp-minimal fields and valuation rings in \cite{johnson21classificationdpmin} (see Section \ref{ss:dpvalfields} for the statements of these results), and a classification of dp-minimal Noetherian domains in \cite{johnson2023}.

In this paper we  continue the arguments of \cite{dp-min-domains,johnson21classificationdpmin,johnson2023},
leading to a classification of dp-minimal integral domains.

\begin{theorem-star}
    Let $R$ be an integral domain.  Then $R$ is dp-minimal if and only if one of the following holds:
    \begin{enumerate}
    \item $R$ is a dp-minimal field.
    \item $R$ is a dp-minimal valuation ring.
    \item There is a valuation subring $\Oo$ of $K = \Frac(R)$, a proper ideal $I \lhd \Oo$, and a finite subring $R_0$ of $\Oo/I$ such that $R$ is the preimage of $R_0$ under the quotient map $\Oo \to \Oo/I$, and the valuation ring $(\Oo,+,\cdot)$ is dp-minimal.
    \end{enumerate}
    Moreover, in (3) the ring $\Oo$, ideal $I$, and subring $R_0$ can be chosen to be definable in $R$.
\end{theorem-star}
See Facts~\ref{dpm-fields-class} and \ref{dpmvr} for the classification of dp-minimal fields and valuation rings.

The ``if'' direction of this theorem---rings satisfying (3) are dp-minimal---is the relatively easy Proposition \ref{prop-construction}.  The ``only if'' direction---any dp-minimal domain arises from one of the three constructions---is the much harder Theorem~\ref{main-thm}.



Here are two typical examples of the construction in (3):
\begin{example} \label{ex-1}
  Let $K = \Qq_3(i)$, where $i = \sqrt{-1}$.  Let $\Oo$ be the natural
  valuation ring on $K$, namely $\Oo = \Zz_3[i]$.  Then $\Oo/27\Oo$ is
  a finite ring extending $\Zz/27\Zz$.  Taking $R_0 = \Zz/27\Zz$, we
  get a dp-minimal integral domain $R \subseteq \Oo$, consisting of
  those $x \in \Oo$ which are congruent mod 27 to an element of
  $\{0,1,2,\ldots,26\}$.  One can describe $R$ more directly as $R =
  \Zz + 27\Oo$.  Note that $R$ is not a valuation ring.
\end{example}

\begin{example} \label{ex-2}
  Let $K$ be the field of Hahn series $\Ff_3^{\alg}((t^\Qq))$, a model
  of ACVF$_{3,3}$.  Let $\Oo$ be the natural valuation ring on $K$,
  and let $I$ be the (non-definable) ideal $I = \{x \in \Oo : v(x) >
  \pi\}$.  There is a finite subring $R_0 \subseteq \Oo/I$ of size
  $3^7$ given by
  \begin{equation*}
    R_0 = \{a + bt + ct^2 + dt^3 + I : a \in \Ff_3, ~ b,c,d \in
    \Ff_9\}.
  \end{equation*}
  This yields a dp-minimal integral domain $R \subseteq \Oo$.
  Explicitly, $R$ is the set of $x \in \Oo$ such that there are $a \in
  \Ff_3$ and $b,c,d \in \Ff_9$ with $v(x-(a+bt+ct^2+dt^3)) > \pi$.  Again, $R$ is not a valuation ring.
\end{example}

During the proof of our main theorem, we prove (or observe) the following results which might be interesting in their own right:

\begin{proposition-star}[Proposition \ref{P:vee valuation ring is ext def}]
If $\Oo$ is a $\vee$-definable valuation subring of some sufficiently saturated field, then $\Oo$ is externally definable.
\end{proposition-star}
\begin{proposition-star}[Proposition \ref{P:ext-val-3}]
  Let $(K,+,\cdot,A)$ be a dp-minimal expansion of a field $K$ by an
  infinite, proper additive subgroup $(A,+) \subseteq (K,+)$.  Then
  there is a non-trivial definable valuation ring on $K$.
\end{proposition-star}
\begin{observation-star}[Fact~\ref{profinite-r/r00}]
    If $R$ is an NIP commutative ring, then $R^{00} = R^0$.
\end{observation-star}
We also make some preliminary observations about more general dp-minimal commutative rings (not domains), showing that each is the product of a finite ring and a dp-minimal henselian local ring, and in dp-minimal local rings, the prime ideals are linearly ordered (Proposition~\ref{P:non-domains}).

\subsection{Sketch of the proof}
In order to prove our main theorem, there are two things to check.  First, we must show that the construction in part (3) of the theorem really does give dp-minimal integral domains.  The proof is relatively easy, using Shelah expansions, and is given in Section~\ref{S:converse}.

The much more difficult direction is to show that if $R$ is a dp-minimal integral domain other than a field or valuation ring, then $R$ arises via the construction in part (3).  The proof of this fact occupies Sections~\ref{S:basic}--\ref{S:mainproof}.  We call such rings $R$ \emph{exceptional dp-minimal domains}.  The fraction field $K = \Frac(R)$ turns out to be a dp-minimal field.  Using the classification of dp-minimal fields, we divide into three cases: $K$ is either \emph{ACVF-like}, \emph{$p$CF-like}, or \emph{RCVF-like} (see Section~\ref{sec:three-kinds}).

Corollary~5.6 in \cite{dp-min-domains} essentially says that $R$ must be comparable with any definable valuation ring on $K$ (see Fact~\ref{F:dp-min R comparable to O}).  Using this, we can quickly rule out the RCVF-like case, at least when $R$ is exceptional (Proposition~\ref{P:acf-pcf-like cover}).

Proposition~3.14 in \cite{dp-min-domains} gives a dichotomy for the behavior of $\sqrt{\mm^{00}}$, the radical of the 00-connected component of the maximal ideal $\mm$ of $R$---it is either $\mm$ itself or the unique second-largest prime ideal in $R$ (see Proposition~\ref{Proposition3.14delbeehalevi} and Lemma~\ref{lem-two-cases}).  We show that these two cases correspond exactly to the ACVF-like case and the $p$CF-like case, respectively (Proposition~\ref{prop-alignment}).  Combining this with \cite[Corollary~5.6]{dp-min-domains} again, we finish the $p$CF-like case in Proposition~\ref{P:construction holds}.

The ACVF-like case is the hardest to deal with.  Assuming that $\Frac(R)$ is ACVF-like, say that $R$ is a \emph{good domain} if $\mm^{00}$ is definable, and a \emph{rogue domain} otherwise (Definition~\ref{def-rogue}).  Good domains are almost trivial to deal with (see the proof of Theorem~\ref{main-thm}); all the work goes into showing that \emph{rogue domains do not exist} (Proposition~\ref{last-prop}).  In Section~\ref{ss:rogue}, we analyze rogue domains, proving enough facts about their structure to get a contradiction.  The proof is technical, but involves showing that the set $\Oo = \mm^{00} : \mm^{00} = \{x \in K : x \mm^{00} \subseteq \mm^{00}\}$ is the coarsest externally definable valuation ring on $K$, and studying the interactions between $R$ and $\Oo$.

The proofs in Section~\ref{S:mainproof} depend on a number of technical tools, which we collect in Section~\ref{S:basic}:
\begin{itemize}
    \item In Subsection~\ref{ss:basicbasic} we review some basic facts about NIP and dp-minimal domains from \cite{dp-min-domains,nip-fp-alg}.
    \item In Subsection~\ref{ss:00-connected} we show certain prime ideals in $R$ are 00-connected.
    \item In Subsection~\ref{ss:ext-def} we show that $\vee$-definable valuation rings are externally definable, and use this to show that $R^{00}$ is externally definable when $R$ is a dp-minimal domain.
    \item In Subsection~\ref{ss:dichotomy} we strengthen and improve the $\sqrt{\mm^{00}}$ dichotomy from \cite[Proposition~3.14]{dp-min-domains}, leveraging the earlier results on 00-connectedness and external definability.
    \item In Subsection~\ref{ss:r-r00} we use tools from topological ring theory to observe that $R^{00} = R^0$ and analyze the structure of $R/R^{00}$ in certain cases using the classification of locally compact fields.
    \item In Subsection~\ref{ss:additive-subgroup} we show that any dp-minimal domain $R$ with an infinite definable additive proper subgroup also admits an infinite definable proper subring (Proposition~\ref{P:A-bdd}).  We subsequently strengthen this to get an externally definable non-trivial valuation ring (Proposition~\ref{P:ext-val-2}) and ultimately a definable non-trivial valuation ring (Proposition~\ref{P:ext-val-3}).
\end{itemize}

\begin{acknowledgment}
CdE is fully supported by the UKRI Horizon Europe Guarantee Scheme, grant No.\@ EP/Y027833/1. YH was supported by ISF grant No. 290/19. WJ was supported
  by the National Natural Science Foundation of China (Grant No.\@
  12101131) and the Ministry of Education of China (Grant No.\@
  22JJD110002). 
\end{acknowledgment}

\section{Preliminaries}\label{S:pre}
\subsection{Notation} \label{ss:notation}
We employ fairly standard model theoretic notation and conventions, see \cite{TZ12,Sim15}.

We briefly review the definition of dp-rank and dp-minimality.  If
$\Sigma(x)$ is a partial type and $\kappa$ is a cardinal, an
\emph{ict-pattern} of depth $\kappa$ in $\Sigma(x)$ is an array of
formulas $(\phi_\alpha(x,b_{\alpha,i}) : \alpha < \kappa, ~ i <
\omega)$ over an elementary extension, such that for every function
$\eta : \kappa \to \omega$, the type
\begin{equation*}
  \Sigma(x) \cup \{\phi_\alpha(x,b_{\alpha,i}) : i = \eta(\alpha)\}
  \cup \{\neg \phi_\alpha(x,b_{\alpha,i}) : i \ne \eta(\alpha)\}
\end{equation*}
is consistent.  The \emph{dp-rank of $\Sigma(x)$}, written
$\dpr(\Sigma(x))$, is the supremum of cardinals $\kappa$ such that
there is an ict-pattern of depth $\kappa$ in $\Sigma(x)$.  (Dp-rank
can also be defined using indiscernible sequences; see \cite[Section
  4.2]{Sim15}.)  The dp-rank of a definable set $D$, written
$\dpr(D)$, is the dp-rank of the formula defining $D$.  A one-sorted
structure $M$ is \emph{dp-minimal} if $\dpr(M) = 1$.  In particular,
$\dpr(M) > 0$, so $M$ must be infinite, with our definition of ``dp-minimal''.  A
one-sorted theory is dp-minimal if its models are dp-minimal.

By a ring we mean a (possibly trivial) commutative ring
with identity.  An ideal in $R$ can be the improper ideal $R$.  A
``maximal ideal'' means a maximal proper ideal.  A local ring is a
ring with a unique maximal ideal, possibly a field.  Valuations can be
trivial, so valuation rings can be fields.

If $(R,+,\cdot,\dots)$ is an expansion of an integral domain then the structures $(R,+,\cdot, \dots)$ and $(\Frac (R),R,+,\cdot,\dots)$ are obviously interdefinable so for example saturation and definability pass from one to the other. We will mostly not distinguish between the two structures.

We assume the reader is familiar with valued fields. Given a valued field $(K,v)$ we will usually denote its value group by $\Gamma$ and its residue field by $k$. Any cut in the value group $\Gamma$ is of the form $\{x\in \Gamma: x>\gamma\}$ or $\{x\in \Gamma: x\geq \gamma\}$ for some $\gamma\in \Gamma'\succ \Gamma$.  Thus any ideal of the valuation ring is of the form $\{x\in K: v(x)\square \gamma\}$, where $\square\in \{>,\geq\}$. An ideal $I$ is principal if and only if $I=\{x\in K:v(x)\geq \gamma\}$ for some $\gamma\in \Gamma$.

We do not assume that we are working in a highly saturated monster model, unless stated otherwise.




\subsection{Dp-minimal valuation rings and fields}\label{ss:dpvalfields}
We recall the classification of dp-minimal valuation rings and fields given by the third author (\cite{johnson21classificationdpmin}).
First recall the following.
\begin{fact}[{\cite[Proposition~5.1]{JSW}}]
  \label{dp-minimal-oag-classification}
  Let $(\Gamma,\le,+)$ be a non-trivial ordered abelian group.  Then $\Gamma$ is
  dp-minimal if and only if $|\Gamma/n\Gamma|$ is finite for all $n >
  0$.
\end{fact}
For example, $\Zz$ and $\Rr$ are dp-minimal as ordered abelian groups.

\begin{fact}
  \label{dpm-fields-class}
  An infinite field $(K,+,\cdot)$ is dp-minimal if and only if there
  is a henselian defectless valuation ring $\Oo \subseteq K$ with
  maximal ideal $\mm$ such that
  \begin{enumerate}
  \item The value group $\Gamma := K^\times/\Oo^\times$ is trivial or dp-minimal as an
    ordered abelian group.
  \item The residue field $k := \Oo/\mm$ is algebraically closed, real
    closed, or $p$-adically closed for some prime $p$.
  \item If the residue field $k$ is algebraically closed of
    characteristic $p > 0$, then the interval $[-v(p),v(p)] \subseteq
    \Gamma$ is $p$-divisible, where $v(p) = +\infty$ when
    $\characteristic(K) = p$.
  \end{enumerate}
\end{fact}
In (2), we mean ``$p$-adically closed'' in the broad sense: $K$ is
$p$-adically closed if $K \equiv K'$ for some finite extension
$K'/\Qq_p$.  Fact~\ref{dpm-fields-class} is proved in
\cite[Theorem~7.3]{johnson21classificationdpmin}, but only under the extra
assumption that $K$ is sufficiently saturated(!).  For completeness, we have
included a proof in the unsaturated case in Appendix~\ref{appendix}.

Next, dp-minimal valuation rings\footnote{The citation is to the classification of dp-minimal \emph{valued fields} $(K,\Oo)$ rather than dp-minimal \emph{valuation rings} $\Oo$, but the distinction does not matter because the two structures $\Oo$ and $(K,\Oo)$ have the same dp-rank by Fact~\ref{F:K-R-rank} below.  Therefore, a valuation ring $\Oo$ is dp-minimal if and only if  the corresponding valued field $(K,\Oo)$ is dp-minimal.} are classified as follows:
\begin{fact}[{\cite[Theorems~1.5, 1.6]{johnson21classificationdpmin}}]
  \label{dpmvr}
  Let $\Oo$ be an infinite valuation ring with fraction field $K$, maximal ideal
  $\mm$, residue field $k = \Oo/\mm$, and value group $\Gamma =
  K^\times/\Oo^\times$.  Then $(\Oo,+,\cdot)$ is dp-minimal if and
  only if the following conditions hold:
  \begin{enumerate}
  \item $(k,+,\cdot)$ and $(\Gamma,+,\le)$ are dp-minimal or finite.
  \item $\Oo$ is henselian and defectless.
  \item One of the following cases holds:
    \begin{itemize}
    \item $k$ is finite of characteristic $p > 0$, $K$ has characteristic 0, and the interval
      $[-v(p),v(p)] \subseteq \Gamma$ is finite.
    \item $k$ is infinite of characteristic $p > 0$, and the interval
      $[-v(p),v(p)] \subseteq \Gamma$ is $p$-divisible.
    \item $k$ has characteristic $0$.
    \end{itemize}
  \end{enumerate}
\end{fact}
For example, $\Zz_p$ and $\Qq_p[[t]]$ are dp-minimal valuation rings, but $\Ff_p[[t]]$ is not.

\section{Constructing dp-minimal integral domains}
\label{S:converse}
Recall that if $M$ is a structure, the \emph{Shelah expansion}
$M^{Sh}$ is the expansion of $M$ by all externally definable sets.  If
$M$ is NIP, then the definable sets in $M^{Sh}$ are precisely the
externally definable sets in $M$, and $M^{Sh}$ is also NIP (see \cite{Shelah783} or \cite[Section 3.3]{Sim15}). 

The following Fact is well-known \cite[Observation~3.8]{OU11}, but we had trouble understanding the proof, so we include a more detailed proof for completeness:

\begin{fact} \label{external-dp}
  $M^{Sh}$ has the same dp-rank as $M$, and in fact, if $D \subseteq
  M^n$ is definable, then $D$ has the same dp-rank whether considered
  in $M$ or in $M^{Sh}$.
\end{fact}
\begin{proof}
  Let $L$ and $L^{Sh}$ be the languages of $M$ and $M^{Sh}$.  Let
  $\Mm^*$ be a monster model elementary extension of $M^{Sh}$, and let
  $\Mm$ be its reduct to $L$.  (Note that $\Mm^*$ is probably not the
  Shelah expansion of $\Mm$.)  Let $\psi(x,c)$ be the $L(M)$-formula
  defining $D$ in $M$ and $M^{Sh}$.  Suppose $\dpr(\psi(x,c)) \ge
  \kappa$ in $\Mm^*$.  We must show $\dpr(\psi(x,c)) \ge \kappa$ in
  $\Mm$.  By definition of dp-rank, there are $L^{Sh}$-formulas
  $\phi_\alpha(x,y)$, elements $a_\eta$ in $\Mm$ for $\eta : \kappa
  \to \omega$, and elements $b_{\alpha,i}$ in $\Mm$ for $\alpha <
  \kappa$ and $i < \omega$, such that
  \begin{gather*}
    \Mm \models \psi(a_\eta,c) \text{ for each } \eta \\
    \Mm^* \models \phi_\alpha(a_\eta,b_{\alpha,i}) \iff \eta(\alpha) = i \text{ for each } \eta, \alpha, i.
  \end{gather*}
  Each $L^{Sh}$-formula $\phi_\alpha(x,y)$ is equivalent \emph{on the
  small set $M$} to an $L$-formula $\theta_\alpha(x,y,e_\alpha)$, with
  $e_\alpha$ in $\Mm$.  Moving the $e_\alpha$'s by an automorphism in
  $\Aut(\Mm^*/M)$, we may assume that $\tp^{L^{Sh}}(\bar{a}\bar{b}/M\bar{e})$
  is finitely satisfiable in $M$.  For any $\eta, \alpha, i$, the $L^{Sh}$-formula
  \begin{equation*}
    \neg(\theta_\alpha(x,y,e_\alpha) \leftrightarrow \phi_\alpha(x,y))
  \end{equation*}
  is not finitely satisfiable in $M$, so $(a_\eta,b_{\alpha,i})$ cannot
  satisfy it, and so
  \begin{equation*}
    \Mm^* \models \theta_\alpha(a_\eta,b_{\alpha,i},e_\alpha) \leftrightarrow \phi_\alpha(a_\eta,b_{\alpha,i}).
  \end{equation*}
  Then
  \begin{gather*}
    \Mm \models \theta_\alpha(a_\eta,b_{\alpha,i},e_\alpha) \iff \eta(\alpha) = i,
  \end{gather*}
  and we have an ict-pattern of depth $\kappa$ in $D$, in the original
  language $L$, showing that $\dpr(D) \ge \kappa$ in $M$.
\end{proof}
In light of Fact~\ref{external-dp}, the following is nearly
trivial:
\begin{proposition} \label{prop-construction}
  Let $\Oo$ be a dp-minimal valuation ring, let $I \subseteq \Oo$ be a
  proper ideal (not necessarily definable), let $R_0$ be a finite
  subring of $\Oo/I$, and let $R \subseteq \Oo$ be the preimage of
  $R_0$ under the quotient map $\Oo \to \Oo/I$.  Then $R$ is a
  dp-minimal integral domain.
\end{proposition}
\begin{proof}
  Let $K$ be the fraction field of $\Oo$, let $\Gamma$ be the value
  group, and let $v : K \to \Gamma$ be the valuation.  As an ideal in
  a valuation ring, $I$ must have the form $\{x \in K : v(x) > \xi\}$
  for some $\xi$ in an elementary extension of $\Gamma$.  The set $\{\gamma \in \Gamma :
  \gamma > \xi\}$ is externally definable, so it is definable in the
  Shelah expansion of $\Oo$.  Then $\Oo/I$, $R_0$, and $R$ are all
  definable in the Shelah expansion.  It follows that $R$ is
  externally definable in $\Oo$.  Then $(\Oo,+,\cdot,R)$ is a reduct
  of the Shelah expansion $\Oo^{Sh}$, so $(\Oo,+,\cdot,R)$ is
  dp-minimal, implying that $(R,+,\cdot)$ is dp-minimal (as $R
  \subseteq \Oo$).
\end{proof}

\begin{remark} \label{remark-on-examples}
  For which dp-minimal valuation rings $\Oo$ can we find an ideal $I$
  and a finite subring $R_0 \subseteq \Oo/I$?  To begin with, $\Oo$
  must have positive residue characteristic, or else $\Oo$ and $\Oo/I$
  are $\Qq$-algebras, which have no finite subrings.  If the residue
  characteristic of $\Oo$ is $p > 0$, then there are two possibilities
  for $\Oo$, by Fact~\ref{dpmvr}(3):
  \begin{enumerate}
  \item The residue field of $\Oo$ is finite.  Then $\Oo$ has mixed
    characteristic and is finitely ramified.  The quotient $\Oo/I$ can
    only have positive characteristic if the ideal $I$ contains $p^n
    \Oo$ for some $n$.  Then $R$ is the pullback of some subring of
    $\Oo/p^n \Oo$.  Conversely, every subring of $\Oo / p^n \Oo$ gives
    a possibility for $R$, because $\Oo / p^n \Oo$ is finite.
    Example~\ref{ex-1} is a typical example.
  \item The residue field of $\Oo$ can be infinite, such as a model of
    $\ACF_p$.  Example~\ref{ex-2} is a typical example.  When $\Oo$ has
    equicharacteristic $p$, there is at least one finite subring of
    $\Oo/I$ for any proper ideal $I$, namely $\Ff_p \subseteq \Oo/I$.
    When $\Oo$ has mixed characteristic, the ideal $I$ must contain
    $p^n \Oo$ for some $n$, but there are many such ideals $I$ because
    the convex hull of $\Zz \cdot v(p)$ is $p$-divisible by
    Fact~\ref{dpmvr}(3).
  \end{enumerate}
  Over the course of the paper, we will see that these two cases
  behave very differently.  For example, note that $R$ has finite
  index in $\Oo$ in the first case, but infinite index in $\Oo$ in the
  second case (because $\Oo/I$ is infinite).  In the first case, $R$
  is definable in the structure $(\Oo,+,\cdot)$, but in the second
  case this can fail (take $I$ non-definable).  Later, we will see
  that $R/R^{00}$ is infinite in the first case, but finite in the
  second case (Proposition~\ref{prop-alignment}).
\end{remark}

\section{Tools}
\label{S:basic}
We review some known and new results on NIP and dp-minimal rings and domains. Unless specified otherwise, any ring below is not assumed to be pure, i.e. there might be some more structure.

\subsection{Basic facts about NIP and dp-minimal domains}
\label{ss:basicbasic}
Every NIP ring has finitely many maximal ideals \cite[Proposition 2.1]{dp-min-domains} and if it is a dp-minimal domain then its prime ideals are linearly ordered by inclusion and so it is either a field or a local ring \cite[Corollary 2.5]{dp-min-domains}.

\begin{fact} \label{F:prime ideals are ext}\label{F:primes are ext def}
    Any prime ideal $\pp$ in a ring $R$ is externally definable.
\end{fact}
When $R$ is NIP, this is \cite[Proposition~2.13]{nip-fp-alg}, but NIP turns out to be unnecessary.
\begin{proof}
    Let $\Sigma(x)$ be the partial type $\{a\mid x : a \in
  R \setminus \pp\} \cup \{b \nmid x : b \in \pp\}$.  Then
  $\Sigma(x)$ is finitely satisfiable---any finite subtype
  \begin{equation*}
    \{a_1 \mid x,~ a_2 \mid x,~ \ldots,~ a_n \mid x,~ b_1 \nmid x,~ b_2 \nmid x,~ \ldots,~ b_m \nmid x\} \subseteq \Sigma(x)
  \end{equation*}
  is realized by the product $a_1a_2 \cdots a_n$.  If $R' \succ R$
  is an elementary extension containing an element $c$ realizing
  $\Sigma(x)$, then the formula $x \nmid c$ defines a set whose
  trace on $R$ is $\pp$.
\end{proof}

 As a consequence, if $(R,+,\cdot,\dots)$ is an NIP ring and $\pp$ is any prime ideal, then the expansion $(R,+,\cdot,\dots,\pp)$ is NIP of the same dp-rank (Fact \ref{external-dp}).  The bi-interpretable structure $(R_\pp,R,+,\cdot,\dots)$ is also NIP, and in fact has the same dp-rank as $R$ by the following:
\begin{fact} \label{F:K-R-rank}
If $R$ is a definable ring and $\pp$ is a definable prime ideal in some structure, then the definable ring $R_\pp$ has the same dp-rank as $R$.  In particular, if $R$ is an integral domain then $\Frac(R)$ has the same dp-rank as $R$.
\end{fact}
This is essentially  \cite[Proposition 2.8(2)]{dp-min-domains} (purity of the structure $R$ is inessential), or \cite[Lemma~10.25]{dpf4} for the case of $\Frac(R)$.




\begin{fact}\label{F:dp-min R comparable to O}
    Let $M=(K,R,\Oo,+,\cdot,\dots)$ be an expansion of a field together with a predicate for a subring and a valuation ring, with $\Frac (R)=\Frac (\Oo)=K$. If $M$ is dp-minimal then either $R\subseteq \Oo$ or $\Oo\subseteq R$.
\end{fact}
This is essentially  \cite[Corollary 5.6]{dp-min-domains}.  Again, purity of the structure $R$ is inessential.

\subsection{00-connectedness of prime ideals}
\label{ss:00-connected}
For any type-definable group $G$ in a sufficiently saturated NIP structure we let $G^{00}$ be its 00-connected component \cite[Section 8.1.3]{Sim15}.
\begin{remark} \label{g00-facts}
    Suppose $G$ is an NIP group, possibly with additional structure.  The following facts are well-known:
\begin{itemize}
      \item If $G$ is sufficiently saturated, then $G^{00}$ is definable if and
      only if $G^{00}$ has finite index in $G$.
      \item If $G$ is sufficiently saturated and $\widehat G \succ G$ is sufficiently saturated, then $G^{00}$ is definable if and only if $(\widehat G)^{00}$ is definable.
      \item We can talk about ``$G^{00}$ is definable'' without
      assuming saturation, by going to a sufficiently saturated elementary
        extension.  The choice of the sufficiently saturated elementary extension
        does not matter by the previous point.
        \item If $G \equiv H$, then $G^{00}$ is definable if and only if $H^{00}$ is definable.
      \item If $G^{00}$ is definable and $H$ is a reduct of $G$ then $H^{00}$ is definable. Note that the converse fails, though---it
        can happen that $G^{00}$ is not definable, but becomes
        definable in a reduct.  For example, if $G$ is the circle
        group in RCF, then $G^{00}$ is not definable, but it becomes
        definable in the pure group reduct.
      \end{itemize}
      More generally, analogous results hold when $G$ is a definable group in an NIP structure $M$.
\end{remark}

      \textbf{From now on, whenever $G$ is some NIP expansion of a group, by ``$G^{00}$ is definable'' we mean that that $(\widehat G)^{00}$ is definable for some/every sufficiently saturated elementary extension $\widehat G\succ G$.}

\begin{fact}[{\cite[proof of Lemma 2.6]{JohDpfiniteIa}}] \phantomsection\label{F:joh19}
\begin{enumerate}
    \item Let $R$ be a sufficiently saturated NIP integral domain and $I\trianglelefteq R$ a type-definable ideal. Then $I^{00}$ is also an ideal of $R$. 
    \item Let $R$ be a sufficiently saturated  NIP local domain with maximal ideal $\mm$. If $R/\mm$ is infinite then $\mm=\mm^{00}$ and more generally, $I = I^{00}$ for any definable ideal $I$.
\end{enumerate}
\end{fact}

We prove a similar result for $\vee$-definable valuation rings.

\begin{proposition} \label{P:00-conn}
Let $K$ be a sufficiently saturated NIP field, $\Oo$ some $\vee$-definable valuation subring and let $\mm$ be its maximal ideal.  Then $\mm$ is type-definable and unless $\mm$ is principal and $\Oo/\mm$ is finite,   $\mm^{00} = \mm$.
\end{proposition}
\begin{proof}
  For non-zero $x\in K$, we have $x \in \mm \iff x^{-1} \notin \Oo$ so $\mm$ is type-definable and $\mm^{00}$ is an  ideal in $\Oo$. Let $v$ be the valuation on $K=\Frac (\Oo)$ associated to $\Oo$.

  Now, break into two cases:
  \begin{description}
  \item[Case 1:] There is a minimal positive element in the value
    group of $\Oo$.  Then $\mm$ is a principal ideal $t\Oo$, where $t$
    is a local uniformizer.  The fact that $\mm = t\Oo$  means that $\mm$ is $\vee$-definable. It follows that $\Oo$ and $\mm$ are definable.  As $\Oo/\mm$ is infinite in this situation, $\mm = \mm^{00}$ by Fact~\ref{F:joh19}(2).
  \item[Case 2:] There is no minimal positive element in the value
    group of $\Oo$.  Suppose for the sake of contradiction that $\mm
    \neq \mm^{00}$.  Take an element $a \in \mm \setminus \mm^{00}$, so $v(a) > 0$.  If $x \in \mm^{00}$, then $x\Oo \subseteq \mm^{00}$ so $a \notin x\Oo$, implfying that $x \in a\Oo$.  This shows that $\mm^{00} \subseteq a\Oo$.
    
    Because there is no minimal positive element in the value group,
    we can find $a_1$ with $0 < v(a_1) < v(a)$.  Then find $a_2$ with
    $0 < v(a_2) < v(a_1)$.  Continuing in this way, we can find a
    sequence
    \begin{equation*}
      a = a_0, a_1, a_2, a_3, \ldots
    \end{equation*}
    with $v(a_0) > v(a_1) > v(a_2) > \cdots > 0$.  Equivalently, there is
    a sequence $a_0, a_1, a_2, \ldots$ with
    \begin{itemize}
    \item $a_0 = a$
    \item $a_i/a_j \in \mm$ for $i < j$
    \item $a_i \in \mm$ for all $i$
    \end{itemize}
    Since $\mm$ is type-definable, these conditions are
    type-definable, and we can use compactness to get a long sequence $\{a_i\}_{i < \kappa}$ satisfying the conditions above.  We get a very long
    ascending chain of subgroups between $\mm^{00}$ and $\mm$:
    \begin{equation*}
      \mm^{00} \subseteq a_0\Oo \subsetneq a_1\Oo \subsetneq a_2\Oo
      \subsetneq \cdots \subseteq \mm.
    \end{equation*}
    But the total number of abstract groups between $\mm^{00}$ and
    $\mm$ is $2^{|\mm/\mm^{00}|}$, so by taking $\kappa >
    2^{|\mm/\mm^{00}|}$ we get a contradiction. \qedhere
  \end{description}
\end{proof} 

\begin{corollary}\label{C:prime-conn}
    Let $R$ be a sufficiently saturated dp-minimal domain. For any non-maximal type-definable prime ideal $\pp$, $R_\pp$ is a henselian valuation ring with maximal ideal $\pp$ and it satisfies $\pp^{00}=\pp$.
\end{corollary}
\begin{proof}
    The fact that $R_\pp$ is a henselian valuation ring with maximal ideal $\pp$ is \cite[Proposition 3.8, Theorem 3.9]{dp-min-domains}. To show connectedness we use Proposition \ref{P:00-conn}: $R_\pp$ is $\vee$-definable by $a\in R_\pp \iff a^{-1}\notin \pp$ so we only need to show that $R_\pp/\pp$ is not finite. Otherwise, $R/\pp$  injects into the field $R_\pp/\pp$ so  the former is a finite integral domain, i.e., a field, which which contradicts non-maximality of $\pp$.
\end{proof}

If $R$ is dp-minimal (or inp-minimal) and $a \in R$ is non-zero, then
there is a unique maximal prime ideal $P_a$ such that $a \notin P_a$
by \cite[Lemma~3.3]{dp-min-domains}.
\begin{corollary} \label{C:Pa-conn}
  If $R$ is a sufficiently saturated dp-minimal domain and $a \in R$
  is non-zero and not a unit, then $P_a$ is a non-maximal
  type-definable prime ideal, and $P_a^{00} = P_a$.
\end{corollary}
\begin{proof}
  Since $a$ is not a unit, it is in the maximal ideal, so $P_a$ must
  be non-maximal.  By \cite[Proposition 3.8, Theorem
    3.9]{dp-min-domains}, $R_{P_a}$ is a valuation ring with maximal
  ideal $P_a$.  By \cite[Remark~3.4]{dp-min-domains}, $R_{P_a}$ is the
  localization $S^{-1}R$ with $S = \{a^n : n \in \Nn\}$.  Then
  $R_{P_a} = \bigcup_{n = 0}^\infty a^{-n}R$, so $R_{P_a}$ is
  $\vee$-definable.  The maximal ideal of a $\vee$-definable valuation
  ring is type-definable, so $P_a$ is type-definable.  The
  00-connectedness of $P_a$ then follows by
  Corollary~\ref{C:prime-conn}.
\end{proof}

\subsection{External definability of $\vee$-definable valuation rings and $R^{00}$}
\label{ss:ext-def}

\begin{proposition}\label{P:vee valuation ring is ext def}
    Let $\Oo$ be a $\vee$-definable valuation subring of some sufficiently saturated field $K$. Then $\Oo$ is externally definable.  If $K$ is NIP, then every ideal of $\Oo$ is externally definable.
\end{proposition}

\begin{proof}
    As observed in the proof of Proposition \ref{P:00-conn}, the maximal ideal $\mm$ of $\Oo$ is type-definable.  By compactness, there is some definable set $D$ with $\mm \subseteq D \subseteq \Oo$.
    
  Let $\Sigma(x)$ be the set of formulas
  
    \[\Sigma(x) = \{x \in aD : a \notin \Oo\} \cup \{x \notin bD : b \in
    \Oo\}.\]

  First suppose that $\Sigma(x)$ is finitely satisfiable.  Then there
  is an element $c$ realizing $\Sigma(x)$ in an elementary extension
  $N$ of the original model.  Then
  \begin{gather*}
    a \notin \Oo \implies c \in a D(N) \implies a \in c D(N)^{-1} \\
    b \in \Oo \implies c \notin bD(N) \implies b \notin cD(N)^{-1}
  \end{gather*}
  Then $\Oo$ is externally definable as the complement of $M \cap
  cD(N)^{-1}$.

  So we may assume that $\Sigma(x)$ is not finitely satisfiable.  Then
  there are $a_1, \ldots, a_n \notin \Oo$ and $b_1,\ldots,b_m \in \Oo$
  such that
  \begin{equation*}
    \bigcap_{i = 1}^n a_iD \subseteq \bigcup_{i = 1}^m b_iD.
  \end{equation*}
  Note that $b_iD \subseteq \Oo$ for each $i$, because $b_i \in \Oo$
  and $D \subseteq \Oo$.  Conversely, for $v$ the valuation associated to $\Oo$, if $x \in \Oo$ then $v(x) \ge
  0$ and $v(a_i) < 0$ (for each $i$), so $v(x/a_i) > 0$ and $x/a_i \in
  \mm \subseteq D$, implying that $x \in a_iD$.  So $\Oo \subseteq
  a_iD$ for each $i$.  Then
  \begin{equation*}
    \Oo \subseteq \bigcap_{i = 1}^n a_iD \subseteq \bigcup_{i = 1}^m
    b_iD \subseteq \Oo.
  \end{equation*}
  Then equality holds, so $\Oo$ equals the definable set $\bigcap_{i =
    1}^n a_iD$.

    The final statement follows since expanding the structure by the externally definable set $\Oo$ does  not change the set of externally definable sets, assuming NIP. 
    Ideals of definable valuation rings are externally definable.
\end{proof}

\begin{lemma} \label{lem-O-vee-def}
  Let $R$ be a sufficiently saturated dp-minimal domain with fraction field
  $K$. We define
  \[\Oo = R^{00} : R^{00} := \{x\in K:xR^{00}\subseteq R^{00}\} \subseteq K.\]  
  Then $\Oo$ is a valuation ring. Further, if $a \in K$,
  then $a \in \Oo$ if and only if $aR/(aR \cap R)$ is finite. In particular, $\Oo=R^{00}:R^{00}$ is a $\vee$-definable and externally definable valuation ring.
\end{lemma}
\begin{proof}
 The proof that $R^{00}:R^{00}$ is a valuation ring is similar to the proof of \cite[Proposition 3.14(1)]{dp-min-domains}.
    \begin{claim}\label{claim_comparabledp-mini}
        $R^{00}:R^{00}$ is a valuation ring.
    \end{claim}
    \begin{claimproof}
        It is routine to show that $R^{00}:R^{00}$ is a ring, so it is sufficient to prove that it is a valuation ring, i.e. for all $a\in K$, $aR^{00}\subseteq R^{00}$ or $a^{-1}R^{00}\subseteq R^{00}$. Let $a\in K$. By \cite[Proposition 4.5]{CKS}, $aR^{00}/(R^{00}\cap aR^{00})$ or $R^{00}/(R^{00}\cap aR^{00})$ is small.  Neither $R^{00}$ nor $aR^{00}$ has any type-definable subgroups of small index, so either $aR^{00} = R^{00} \cap aR^{00}$ or $R^{00} = R^{00} \cap aR^{00}$.  Equivalently, $aR^{00} \subseteq R^{00}$ or $R^{00} \subseteq aR^{00}$, i.e., $a^{-1}R^{00} \subseteq R^{00}$.
    \end{claimproof} 
If $a\in \Oo$ then $aR^{00}\subseteq R^{00}$ so $aR^{00}\subseteq aR\cap R\subseteq aR$, so $aR\cap R$ has bounded index in $aR$ but both are definable so finite index. If $aR\cap R$ has finite index in $aR$ then $(aR)^{00}\subseteq aR\cap R\subseteq R$, but $(aR)^{00}=aR^{00}$ so $aR^{00}\subseteq R$ and $aR^{00}=(aR^{00})^{00}\subseteq R^{00}$. 

It follows that $\Oo$ is defined by the disjunction of conditions of the form ``$aR/(aR\cap R)$ is finite". By Proposition \ref{P:vee valuation ring is ext def}, $\Oo$ is externally definable.
\end{proof}

\begin{corollary}\label{C:R00:R00}
    Let $R$ be a sufficiently saturated dp-minimal domain and $K=\Frac (R)$ and let $\Oo=R^{00}:R^{00}$. Then $R^{00}$ is an externally definable ideal of $\Oo$, and $R/R^{00}$ is a dp-minimal ring. 
    
    If, furthermore, $R^{00}$  is a non-maximal prime ideal of $R$ then $R^{00}$ is the maximal ideal of $\Oo$ and $\Oo$ is equal to the localization of $R$ at $R^{00}$.
\end{corollary}
\begin{proof}
    If $R$ is a field then $R=R^{00}$ and we have nothing to show; so assume otherwise. Then $R^{00}$ is an ideal of the valuation ring $\Oo$ so by Lemma \ref{lem-O-vee-def} and Proposition \ref{P:vee valuation ring is ext def}, it is also externally definable.

    Thus adding $R^{00}$ as a predicate preserves dp-minimality (Fact~\ref{external-dp}) so $R/R^{00}$ is dp-minimal as well.

    If $R^{00}$ is a non-maximal prime ideal $\pp$ of $R$ then by Corollary \ref{C:prime-conn} the localization $R_\pp$ is a valuation ring with maximal ideal $\pp$.  Note that $\Oo' = \mm' : \mm'$ for any valuation ring $(\Oo',\mm')$, and so
    \begin{equation*}
      \Oo := R_{00} : R_{00} = \pp : \pp = R_\pp.
    \end{equation*}
    Thus $\Oo$ is the localization $R_\pp$ and $R^{00}$ is the maximal
    ideal of $\Oo$.
\end{proof}

\subsection{The $\sqrt{\mm^{00}}$ dichotomy}
\label{ss:dichotomy}

We observe that the proof of \cite[Proposition~3.14]{dp-min-domains} had a gap\footnote{In the proof of \cite[Proposition~3.14]{dp-min-domains}, we misunderstood the notation ``$< \infty$'' in \cite[Proposition~3.12]{chain-conditions} and \cite[Proposition 4.5(2)]{CKS} to mean ``finite'', when it in fact meant ``small''.}, which we correct now.
\begin{fact}\label{almost-comparable}
Let $G$ be an abelian dp-minimal group in a structure $M$.
\begin{enumerate}
\item If $M$ is sufficiently saturated and $X, Y$ are type-definable subgroups of $G$, then $X/(X \cap Y)$ or $Y/(X \cap Y)$ is small (relative to the degree of saturation).
\item If $X, Y$ are definable subgroups of $G$, then $X/(X \cap Y)$ or $Y/(X \cap Y)$ is finite.
\item If $X, Y$ are externally definable subgroups of $G$, then $X/(X \cap Y)$ or $Y/(X \cap Y)$ is finite.
\end{enumerate}
\end{fact}
Part (1) is \cite[Proposition 4.5(2)]{CKS}.  It easily implies (2), which then implies (3) by passing to the Shelah expansion.

\begin{proposition}\label{Proposition3.14delbeehalevi}\cite[Proposition~3.14]{dp-min-domains} 
Let $R$ be a sufficiently saturated dp-minimal domain with maximal ideal $\mm$ and fraction field $K$. Then
\begin{enumerate}
    \item $\mm^{00}:\mm^{00}$ is a valuation overring of $R$;
    \item $\sqrt{\mm^{00}} = \set{a\in \mm\mid R/aR\text{ is infinite}}$;
    \item For every $a\in \mm$ with $R/aR$ finite, $\sqrt{\mm^{00}} = P_a$.
    \end{enumerate} 
    In particular, exactly one of the following holds:
    \begin{itemize}
        \item $R/aR$ is infinite for all $a\in \mm$ and in this case $\sqrt{\mm^{00}}=\mm$; 
        \item There is $a \in \mm$ with $R/aR$ finite, and for any such $a$, $\sqrt{\mm^{00}}=P_a$.
    \end{itemize}

\end{proposition}
\begin{proof}
Before beginning the proof, we make a general observation about local rings: if $I \subseteq J \subsetneq R$ are proper ideals in a local ring $R$, and $J/I$ is finite, then $J \subseteq \sqrt{I}$.  Indeed, for any $b \in J$, the pigeonhole principle gives $b^n \equiv b^m \pmod{I}$ for some $n < m < \omega$.  Then $b^n(1 - b^{m-n}) \in I$, implying that $b^n \in I$, because $1 - b^{m-n}$ is a unit.  In particular, if $\pp \subsetneq \qq$ are prime ideals, then $\qq/\pp$ is infinite, or else $\qq \subseteq \sqrt{\pp} = \pp$.  (Compare with \cite[Remark~3.7]{dp-min-domains}.)

\textit{1.} It is routine to show that $\mm^{00}:\mm^{00}$ is a ring, so it is sufficient to prove that it is a valuation ring, i.e. for all $a\in K$, $a\mm^{00}\subseteq \mm^{00}$ or $a^{-1}\mm^{00}\subseteq \mm^{00}$. This follows immediately by the same argument as in Claim \ref{claim_comparabledp-mini}.


\textit{2.} If $R/\mm$ is infinite then $\mm=\mm^{00}$ by Fact \ref{F:joh19} and the equality is obvious. We assume that $R/\mm$ is finite and show that $a\in \mm\setminus \sqrt{\mm^{00}}$ if and only if $R/aR$ is finite.   When $R/\mm$ is finite, note that $\mm^{00} = R^{00}$, and so $\mm^{00}$ is externally definable by Corollary~\ref{C:R00:R00}.

Let $a\in \mm\setminus \sqrt{\mm^{00}}$. By Fact~\ref{almost-comparable}(3) we have that $aR/(\mm^{00}\cap aR)$ is finite or $\mm^{00}/(\mm^{00}\cap aR)$ is finite. If it was the former, then by the remarks about local rings at the start of the proof,
\[ a \in aR \subseteq \sqrt{\mm^{00} \cap aR} \subseteq \sqrt{\mm^{00}},\]
a contradiction.  Consequently, $\mm^{00}/(\mm^{00}\cap aR)$ is finite and by the definition of $\mm^{00}$, $\mm^{00}\subseteq aR$ and $|\mm/aR|<\infty$ hence $|R/aR|=|R/\mm||\mm/aR|<\infty$.

Conversely, if $R/aR$ is finite, then \[\mm^{00} = R^{00} = (aR)^{00} = aR^{00} = a\mm^{00}.\] Suppose for the sake of contradiction that $a \in \sqrt{\mm^{00}}$.  Take $n$ minimal such that $a^n \in \mm^{00}$.  Then $a^n \in \mm^{00} = a\mm^{00}$, implying $a^{n-1} \in \mm^{00}$, which contradicts the choice of $n$.

\textit{3.} Assume that $R/aR$ is finite. By \textit{2}, $a\notin \sqrt{\mm^{00}}$ so $a \notin \pp$ for some prime ideal $\pp \supseteq \mm^{00}$.  By choice of $P_a$, we have $\sqrt{\mm^{00}} \subseteq \pp \subseteq P_a$. For the other inclusion, note that $P_a \subsetneq \mm$, and so $\mm/P_a$ is infinite by the remarks on local rings at the start of the proof.  If $b \in P_a$, then $bR \subseteq P_a \subsetneq \mm \subseteq R$, so $R/bR$ is infinite, and $b \in \sqrt{\mm^{00}}$ by \textit{2}.  This shows the reverse inclusion $P_a \subseteq \sqrt{\mm^{00}}$.
\end{proof}

\begin{corollary}\label{C:when m=m00}
Let $R$ be a sufficiently saturated dp-minimal domain with maximal ideal $\mm$ and fraction field $K$. Then $\mm=\mm^{00}$ if and only if $\mm^{00}$ is prime and $R/aR$ is infinite for all $a\in \mm$.
\end{corollary}
\begin{proof}
Assume that $\mm=\mm^{00}$; so it is prime. Assume further that $R/aR$ is finite for some $a\in \mm$. By Proposition~\ref{Proposition3.14delbeehalevi}, $\mm=P_a$ so the latter is a maximal ideal which gives that $a$ is a unit, contradicting the choice of $a$. So $R/aR$ is infinite for any $a \in \mm$.

If $\mm^{00}$ is prime and $R/aR$ is infinite for all $a\in \mm$ we conclude that $\mm^{00}=\mm$ by Proposition~\ref{Proposition3.14delbeehalevi}.
\end{proof}

Corollary~\ref{C:Pa-conn} yields a strengthening of Proposition \ref{Proposition3.14delbeehalevi}.

\begin{lemma}\label{lem-two-cases}
  Let $R$ be a sufficiently saturated dp-minimal domain with maximal ideal
  $\mm$.
  \begin{enumerate}
  \item If $R/aR$ is infinite for every $a \in \mm$, then
    $\sqrt{\mm^{00}} = \mm$.
  \item If $a \in \mm$ and $R/aR$ is finite, then $\mm^{00} = P_a$,
    and $\mm^{00}$ is the largest non-maximal prime ideal.
  \end{enumerate}
\end{lemma}
\begin{proof}
  Proposition \ref{Proposition3.14delbeehalevi} proves (1), and shows that
  in setting of (2), we have $P_a = \sqrt{\mm^{00}}$.  But $P_a = P_a^{00}$ by Corollary~\ref{C:Pa-conn}, and so
  \begin{equation*}
    P_a = P_a^{00} \subseteq \mm^{00} \subseteq \sqrt{\mm^{00}} = P_a.
  \end{equation*}
  Equality must hold, and $\mm^{00} =
  P_a$.  In particular, $\mm^{00}$ is a non-maximal prime ideal.  If $\mm^{00}$
  fails to be the largest non-maximal prime ideal, then there is a
  prime ideal $\pp$ with $\mm^{00} \subsetneq \pp \subsetneq \mm$.
  Take $b \in \mm \setminus \pp$.  Enlarging $\pp$, we may assume
  $\pp$ is the largest prime ideal not containing $b$, i.e., $\pp =
  P_b$.  Then
  \begin{equation*}
    P_b = P_b^{00} \subseteq \mm^{00} \subsetneq P_b,
  \end{equation*}
  by Corollary~\ref{C:Pa-conn} and the fact that $P_b \subseteq \mm$.
  This is absurd, so $\mm^{00}$ is truly the second-largest prime
  ideal.
\end{proof}

\subsection{The topological ring $R/R^{00}$} \label{ss:r-r00}
Let $R$ be an NIP ring.

The ``logic topology'' on $R/R^{00}$ is the topology where a set $X
\subseteq R/R^{00}$ is closed if its preimage $\tilde{X} \subseteq R$
is type-definable.  The logic topology makes $R/R^{00}$ into a compact
Hausdorff topological ring; see \cite[Section 8.1]{Sim15}.

In contrast to compact Hausdorff topological groups, the following holds:
\begin{fact}[{\cite[Proposition 5.1]{profinitegroups}}]
    A topological ring is compact Hausdorff if and only if it is profinite.
\end{fact}

As a direct consequence: 

\begin{fact}\label{profinite-r/r00}
    If $R$ is an NIP ring then $R/R^{00}$ is profinite and so $R^{00}=R^0$.
\end{fact}

The following  should be well known, but we include a proof for completeness.
\begin{fact}\label{F:profinite domain}
    Let $A$ be an infinite compact Hausdorff integral domain such that every non-trivial principal ideal has finite index. Then $F=\Frac (A)$ is either isomorphic to a finite extension of $\mathbb{Q}_p$ or to $\mathbb{F}_q((t))$.
\end{fact}
\begin{proof}
  We first make some observations.
  \begin{enumerate}
  \item If $a \in A \setminus \{0\}$, then the principal ideal $aA$ is closed because it is the image of the compact set $A$ under the continuous function $f(x) = ax$.  Since $aA$ has finite index, $aA$ is a clopen ideal.
  \item Every non-zero ideal $I$ contains a principal ideal, and
    therefore is clopen and has finite index.
  \item Conversely, the zero ideal is not open, since $A$ is infinite.
  \item The set of units $A^\times$ is closed, being the projection of the compact set $\{(x,y) \in A^2 : xy=1\}$ onto the first coordinate.  If $A$ is a field, this contradicts the previous point, so $A$ is not a field.      
  \item By \cite[Proposition 5.1.2]{profinitegroups}, the open ideals
    form a neighborhood basis of $0$.  By the points above, the non-zero ideals of $A$ form a neighborhood basis of $0$, and the family
    \begin{equation*}
      \{I+a : I \trianglelefteq A \text{ a non-zero ideal},\, a\in A\}
    \end{equation*}
    is a basis for the topology.
  \item Every commutative profinite ring is a cartesian product of profinite local rings \cite[Exercise 5.1.3(2)]{profinitegroups} so as $A$ is an integral domain it must be a local profinite ring.
  \end{enumerate}
    Let $F=\Frac (A)$.    By \cite[Example 1.2]{PrZi}, the family
    \begin{equation*}
      \{I+a : I \trianglelefteq A \text{ a non-zero ideal},\, a\in F\}
    \end{equation*}
    is a basis for a non-discrete Hausdorff ring topology on $F$.
    Moreover, this ring topology is a field topology because $A$ is
    local; see the proof of \cite[Theorem 2.2(b)]{PrZi}.  Taking $I =
    A$ we see that $A$ is clopen in this topology.  By the
    observations above, the induced topology on $A$ is the original
    compact Hausdorff topology on $A$.  Then $F$ is a locally compact
    field, because it is covered by the compact open translates $A+b$.

    As $F$ is a locally compact field, it is isomorphic to $\mathbb{R}$, $\mathbb{C}$, a finite extension of $\mathbb{Q}_p$, or $\mathbb{F}_q((t))$ \cite[Remark 7.49]{Milne}. Since $\mathbb{R}$ and $\mathbb{C}$ have no compact subrings it must be one of the latter two.
\end{proof}

\begin{corollary}\label{C:R/R00 is Qp in some NIP}
    Let $R$ be a sufficiently saturated NIP integral domain with $R^{00}$ a prime ideal satisfying $R^{00}=aR^{00}$ for any $a\in R\setminus R^{00}$.  Then $\Frac \left( R/R^{00}\right)$ is either finite or a finite extension of $\mathbb{Q}_p$.
\end{corollary}
\begin{proof}
    Since $R^{00}$ is a prime ideal, $R/R^{00}$ is a compact Hausdorff integral domain; assume it is infinite. Note that every nonzero principal ideal in $R/R^{00}$ has finite index.  Indeed if $a\in R\setminus R^{00}$ then $aR\supseteq (aR)^{00}=aR^{00}=R^{00}$, so $aR$ has finite index in $R$. By Fact \ref{F:profinite domain}, $\Frac (R/R^{00})$ is either a finite extension of $\mathbb{Q}_p$ or $\mathbb{F}_q((t))$. Since $R^{00}$ is a prime ideal it is externally definable, by Fact \ref{F:prime ideals are ext}.  Then $R/R^{00}$ and hence $\Frac (R/R^{00})$ is NIP.  Since $\Ff_q((t))$ has IP \cite[Theorem~4.3]{KSW}, $\Frac(R/R^{00})$ must instead be a finite extension of $\Qq_p$.
\end{proof}

\begin{corollary}\label{C:R/R00 is Qp in dp-min}
    Let $R$ be a sufficiently saturated dp-minimal integral domain with maximal ideal $\mm$. If $R/\mm$ is finite and $R^{00}$ is a prime ideal then  $\Frac \left( R/R^{00}\right)$ is either finite or a finite extension of $\mathbb{Q}_p$.
\end{corollary}
\begin{proof}
    If $R/R^{00}$ is finite then it is a finite integral domain hence a finite field. Assume that $R/R^{00}$ is infinite.

    In particular, $\pp=R^{00}$ is a non-maximal prime ideal so the maximal ideal of the valuation overring $R_\pp$ of $R$ (Corollary \ref{C:prime-conn}). This implies that $aR^{00}=R^{00}$ for any $a\in R_\pp\setminus R^{00}$, hence in particular for all $a\in R\setminus R^{00}$. The result now follows from Corollary \ref{C:R/R00 is Qp in some NIP}.
\end{proof}

\subsection{Subrings from additive subgroups} \label{ss:additive-subgroup}
In this section, we prove the existence of an infinite definable proper subring in dp-minimal
structures $(K,+,\cdot,A)$, where $(K,+,\cdot)$ is
a field and $A$ is an infinite  additive
proper subgroup, without any saturation assumption.  This will be used later.

\begin{proposition} \label{P:A-bdd}
  Let $M=(K,+,\cdot,A)$ be a dp-minimal expansion of field with $A$ a proper infinite additive subgroup.  Then $R=A:A=\{x\in K:xA\subseteq A\}$ is an infinite definable proper subring of $K$ and $\Frac (R)= K$.
\end{proposition}
\begin{proof}
  There is no harm in assuming sufficient saturation of $M$, also we will later pass to the Shelah expansion; this too causes no harm.
  
  The definable set $A$ is neither finite nor cofinite, so $(K,+,\cdot,A)$ is not strongly minimal.

  By \cite[Theorem~1.3]{JohCanonical},  $M$ has a Hausdorff non-discrete definable V-topology called the \emph{canonical topology}. The canonical topology is characterized by the fact that \[ \{D - D : D \subseteq K \text{ definable and infinite}\}\] is a neighborhood basis of 0, where $D-D = \{x-y : x,y \in D\}$.  Taking $D = A$, we see that $A = A-A$ is a neighborhood of 0.
  
  By saturation, there exists an externally definable valuation inducing this topology (\cite[Proposition 3.5]{HaHaJaVtop}). So after passing to the Shelah expansion, we may assume the existence of a definable valuation $v$ inducing this topology. Let $\Gamma$ be its value group.

  \begin{claim}
      $A$ is a bounded set with respect to topology induced by $v$.
  \end{claim}
  \begin{claimproof}
     Note that $(K,+)^0$ is an ideal in $K$ of small index, so $(K,+)^0 =(K,+)$ and so $K$ has no definable proper subgroups of finite index.
   Therefore $K/A$ is infinite. Let $(c_i)_{i<\omega}$ be different representatives of distinct classes.
  
      Assume towards a contradiction that $A$ is not bounded, so for any $\gamma\in \Gamma$ there exists an element  $a\in A$ with $v(a)<\gamma$.

    By saturation, and the fact that A is infinite, we can find a sequence of elements  $(a_s)_{s<\omega}$ of $A$, satisfying $v(a_s)<v(a_t)<v(c_i)$ for any $i<\omega$ and any $t<s<\omega$.

    Now consider the formulas $(x\in A+c_i)_{i<\omega}$ and $(v(x)=v(a_s))_{s<\omega}$; they contradict dp-minimality.
  \end{claimproof}
  
  The definable set $R$ is certainly a definable subring of $K$, because $A$ is an additive subgroup. 
  
  Since $A$ is a bounded neighborhood of 0, there is an open neighborhood $V \ni 0$ such that $V \cdot A \subseteq A$.  Then $R$ contains the infinite set $V$.

  The ring $R$ is a proper subring because if we take $a\in A\setminus \{0\}$ and $b\notin A$ then $b/a\notin R$. Finally, $\Frac (R) =K$, since $(K,\Frac (R))$ is dp-minimal so $[K:\Frac (R)]=1$.
  \end{proof}

\section{A trichotomy for dp-minimal fields} \label{sec:three-kinds}
\begin{definition}
  Let $K$ be an infinite field.
  \begin{enumerate}
  \item $K$ is \emph{ACVF-like} if there is a henselian valuation $v$
    on $K$ with algebraically closed residue field.
  \item $K$ is \emph{$p$CF-like} if there is a henselian valuation $v$
    on $K$ with finite residue field of
    characteristic $p$.
  \item $K$ is \emph{RCVF-like} if there is a henselian valuation $v$
    on $K$ with real closed residue field.
  \end{enumerate}
  We allow $v$ to be trivial.
\end{definition}
For example, $\Ff_p^{\alg}$ and $\Cc((t))$ are ACVF-like, $\Rr$ and
$\Rr((t))$ are RCVF-like, and $\Qq_3(i)$ and $\Qq_3((t))$ are
$3$CF-like.

If there is a henselian valuation $w$ on $K$ such that $Kw$ is
$p$-adically closed, then we can compose $w$ with the canonical
valuation on $Kw$, getting a henselian valuation $v$ on $K$ with $Kv$
finite of characteristic $p$.  Thus $K$ is $p$CF-like.

\begin{remark} \label{rem:trichotomy}
  If $K$ is a dp-minimal field, then $K$ is ACVF-like, $p$CF-like, or
  RCVF-like.  In fact, if $\Oo_\infty$ is the intersection of all
  0-definable valuation rings on $K$, then $\Oo_\infty$ is henselian
  and its residue field is algebraically closed, real closed, or
  finite.  See \cite[Theorem~4.8]{johnson21classificationdpmin} or Fact~\ref{F:appendix}.
\end{remark}

In fact, the three cases are mutually exclusive:
\begin{proposition} \label{three-cases}
  If $K$ is an infinite field, then $K$ satisfies at most one of the following:
  \begin{enumerate}
  \item $K$ is ACVF-like.
  \item $K$ is $p$CF-like.
  \item $K$ is RCVF-like.
  \end{enumerate}
 In the second case, $p$ is uniquely determined.
\end{proposition}
\begin{proof}
  Otherwise, there are two henselian valuation rings $\Oo_1, \Oo_2$
  whose residue fields $k_1, k_2$ are of a different nature.  There
  are two cases:
  \begin{enumerate}
  \item The two valuations are comparable, say, $\Oo_1 \subseteq
    \Oo_2$.  Then $\Oo_2$ is a coarsening of $\Oo_1$, so $k_1$ is the
    residue field of some henselian valuation $w$ on $k_2$.  If $k_2$
    is finite, then $w$ is trivial, so $k_1 = k_2$ and $k_1$ is finite
    of the same characteristic as $k_1$.  If $k_2$ is algebraically
    closed or real closed, then $(k_2,w)$ is a model of ACVF or RCVF,
    respectively, and so $k_1$ is algebraically closed or real closed.
    In each case, $k_1$ has the same nature as $k_2$.
  \item The two valuations are incomparable.  Let $\Oo_3$ be the join
    $\Oo_1 \cdot \Oo_2$ and let $k_3$ be its residue field.  Then
    $\Oo_1$ and $\Oo_2$ induce independent henselian valuations $w_1,
    w_2$ on $k_3$, and the residue fields of $w_1$ and $w_2$ are $k_1$
    and $k_2$, respectively.  Because $k_3$ has two independent
    nontrivial henselian valuations, it must be separably closed
    \cite[Theorem~4.4.1]{PrEn}, and then the two residue fields $k_1$
    and $k_2$ must be algebraically closed
    \cite[Theorem~3.2.11]{PrEn}.  Then $k_1$ and $k_2$ again have the
    same nature.  \qedhere
  \end{enumerate}
\end{proof}

Moreover, everything
is determined by the complete theory of $(K,+,\cdot)$:
\begin{lemma}
  Let $K_1, K_2$ be two dp-minimal fields.  If $K_1 \equiv K_2$, then
  $K_1$ is ACVF-like (RCVF-like, $p$CF-like) if and only if $K_2$ is
  ACVF-like (RCVF-like, $p$CF-like).
\end{lemma}
\begin{proof}
  Let $\Oo_i$ be a henselian valuation ring on $K_i$ whose residue
  field is algebraically closed, real closed, or finite.  By Robinson
  joint consistency, there is a structure $(K_3,\Oo_1',\Oo_2')$ with
  elementary embeddings $(K_i,\Oo_i) \hookrightarrow (K_3,\Oo_i')$ for
  $i = 1, 2$.  In particular, $\Oo_i'$ is a henselian valuation ring
  on $K_3$ whose residue field has the same nature as the residue
  field of $\Oo_i$.  By Proposition~\ref{three-cases}, the residue fields
  of $\Oo_1'$ and $\Oo_2'$ must have the same nature.  Therefore, the
  residue fields of $\Oo_1$ and $\Oo_2$ must have the same nature.
\end{proof}


\begin{remark} \label{ocan-00}
  Let $K$ be a sufficiently saturated dp-minimal field, possibly with extra structure.  Suppose that $K$ is $p$CF-like.
  Let $\Oo_{\can}$ be the intersection of all definable valuation
  rings on $K$.  By \cite[Theorem~4.7]{johnson21classificationdpmin},
  $\Oo_{\can}$ is henselian, definable, and has finite residue field,
  because $K$ is $p$CF-like rather than ACVF-like or
  RCVF-like.\footnote{The notation reflects the fact that $\Oo_{\can}$
  is the canonical henselian valuation ring on $K$.}  We claim that
  $\Oo_{\can}^{00}$ is the second largest prime ideal in $\Oo_{\can}$, and that $\Oo_{\can}/\Oo_{\can}^{00}$ is a ring of characteristic 0:

  The valuation ring $\Oo_{\can}$ is dp-minimal with finite residue
  field, so the interval $[-v(p),v(p)]$ in the value group of
  $\Oo_{\can}$ is finite by Fact~\ref{dpmvr}.  Then the value group is
  discretely ordered and so the maximal ideal $\mm_{\can}$ of
  $\Oo_{\can}$ is a principal ideal.  Since $\Oo_{\can}/\mm_{\can}$ is finite, $\Oo_{\can}^{00} = \mm_{\can}^{00}$ is
  the second largest prime ideal in $\Oo_{\can}$ by
  Lemma~\ref{lem-two-cases}.
  
  Because $\Oo_{\can}^{00}$ is a non-maximal prime
  ideal in $\Oo_{\can}$, the quotient $\Oo_{\can}/\Oo_{\can}^{00}$
  is an integral domain but not a field.  In particular, it is
  infinite.  By Corollary~\ref{C:R/R00 is Qp in dp-min}, the fraction
  field of $\Oo_{\can}/\Oo_{\can}^{00}$ must be a finite extension of
  $\Qq_p$, and must therefore have characteristic 0.
\end{remark}

\section{Proof of the main theorem}
\label{S:mainproof}

\subsection{Exceptional dp-minimal domains} \label{ss:exceptional}
An integral domain which is neither a field nor a valuation ring will be called \emph{exceptional}.  Our goal is to understand exceptional dp-minimal domains.  By \cite[Lemma 3.6(1)]{dp-min-domains}, any exceptional dp-minimal domain $R$ is a local domain with finite residue field $R/\mm$ . In particular, for such domains, $R^{00}=\mm^{00}$.  Note also that $R$ is infinite.

\begin{remark} \label{R:r-vs-o-exceptional}
  Let $R$ be an exceptional dp-minimal domain (possibly with extra structure).  If $\Oo$ is an externally definable valuation ring on $\Frac(R)$, then $R \subseteq \Oo$.  First, we may assume that $\Oo$ is definable by passing to $R^{Sh}$.  Then Fact \ref{F:dp-min R comparable to O} shows that $R \subseteq \Oo$ or $\Oo \subseteq R$.  The latter cannot happen as $R$ is not a valuation ring.
\end{remark}



Recall that every infinite dp-minimal field is either ACVF-like,
RCVF-like or $p$CF-like, and only one of these properties holds
(Remark~\ref{rem:trichotomy} and Proposition~\ref{three-cases}).
\begin{proposition}\label{P:acf-pcf-like cover}
  If $R$ is an exceptional dp-minimal domain, then $\Frac(R)$ is ACVF-like or $p$CF-like.
\end{proposition}
\begin{proof}
  We may assume that $R$ is sufficiently saturated.  The fraction field $K = \Frac(R)$ is dp-minimal by Fact~\ref{F:K-R-rank}.  Suppose for the sake of contradiction that $K$ is RCVF-like.  Then $K$ admits a henselian valuation with real closed residue field.  By Ax-Kochen-Ershov, $K$ is elementarily equivalent to the Hahn series $L=\mathbb{R}((t^\Gamma))$ for some dp-minimal ordered abelian group $\Gamma$.  This field $L$ is orderable.  By \cite[Proof of Corollary 6.2]{JSW}, any field-order on $L$ is  definable. As a result there is a definable field-ordering $\leq$ on $K$ as well.  Then $(K,R,\leq)$ is dp-minimal. Let $\Oo$ be the convex hull of $\mathbb{Z}$ in $(K,\leq)$; it is an externally definable  valuation subring of $K$.  By Remark~\ref{R:r-vs-o-exceptional}, $R\subseteq \Oo$. By saturation, $R\subseteq [-n,n]$ for some integer $n$, which is absurd.
\end{proof}
The next proposition shows that the two cases of
Lemma~\ref{lem-two-cases} correspond to the case where $\Frac(R)$ is
ACVF-like and $p$CF-like.
\begin{proposition} \label{prop-alignment}
  Let $R$ be an exceptional dp-minimal domain, with maximal ideal
  $\mm$.
  \begin{enumerate}
  \item If $R/aR$ is infinite for every $a \in \mm$, then $\Frac(R)$
    is ACVF-like.
  \item If $R/aR$ is finite for some $a \in \mm$, then $\Frac(R)$ is
    $p$CF-like.
  \end{enumerate}
\end{proposition}
\begin{proof}
  Since $\Frac(R)$ is either ACVF-like or $p$CF-like by
  Proposition~\ref{P:acf-pcf-like cover}, it suffices to show that
  $\Frac(R)$ is $p$CF-like if and only if $R/aR$ is finite for some
  $a \in \mm$. There is no harm in assuming that $R$ is sufficiently saturated.

  First suppose that $\Frac(R)$ is $p$CF-like.  Let $\Oo_{\can}$ be
  the canonical henselian valuation ring, as in Remark~\ref{ocan-00}.  Then $\Oo_{\can}$ is a definable henselian
  valuation ring with finite residue field, the second largest prime
  ideal of $\Oo_{\can}$ is $\Oo_{\can}^{00}$, and the quotient ring
  $\Oo_{\can}/\Oo_{\can}^{00}$ has characteristic zero.

  By Remark~\ref{R:r-vs-o-exceptional}, as $R$ is exceptional we
  have $R\subseteq \Oo_{\can}$ hence $R^{00}\subseteq
  \Oo_{\can}^{00}\cap R$.  Note that $\Oo_{\can}^{00}\cap R$ is a
  prime ideal of $R$.  The quotient $R/(R \cap \Oo_{\can}^{00})$
  injects into $\Oo_{\can}/\Oo^{00}_{\can}$, so it has characteristic
  zero.  In contrast, $R/\mm$ is finite.  Therefore $R \cap
  \Oo_{\can}^{00} \ne \mm$.  Since $\mm^{00} = R^{00}$ is contained in
  the non-maximal prime ideal $R \cap \Oo_{\can}^{00}$, we have
  $\sqrt{\mm^{00}} \subsetneq \mm$.  By Lemma~\ref{lem-two-cases}(1),
  $R/aR$ is finite for some $a \in \mm$.

  Conversely, suppose that $R/aR$ is finite for some $a \in \mm$.
  By Lemma~\ref{lem-two-cases}(2), $R^{00} = \mm^{00}$ is a
  non-maximal prime ideal of $R$.  Then $k = \Frac(R/R^{00})$ is a
  finite extension of $\Qq_p$ by Corollary \ref{C:R/R00 is Qp in
    dp-min}. Let $\Oo_p$ be the standard definable valuation ring of $k$.

  Let $K=\Frac(R)$ and let $\Oo=R^{00}:R^{00}$; it is an externally definable valuation overring with maximal ideal $R^{00}$ (Corollary \ref{C:R00:R00}) so $\Oo/R^{00}$ is its residue field. Now using the quotient map, the structure $(\Oo/R^{00}, R/R^{00})$ is dp-minimal hence so is $(\Oo/R^{00}, \Frac(R/R^{00}))$. As $R/R^{00}$ is infinite, necessarily $\Frac(R/R^{00})=k$ must be equal to $\Oo/R^{00}$. Composing the valuation $\Oo_p$ on $k$ with $\Oo$ we are supplied with an externally definable valuation on $K$ with finite residue field; so $K = \Frac(R)$ is pCF-like.
\end{proof}
Combining with Lemma~\ref{lem-two-cases}, we get the following
dichotomy for exceptional dp-minimal domains:
\begin{corollary} \label{cor-two-cases}
  Let $R$ be an exceptional dp-minimal domain, with maximal ideal
  $\mm$.  Then exactly one of the following holds:
  \begin{enumerate}
  \item $\Frac(R)$ is ACVF-like, $R/aR$ is infinite for every $a \in
    \mm$, and $\sqrt{\mm^{00}} = \mm$.
  \item $\Frac(R)$ is $p$CF-like, $R/aR$ is finite for some $a \in
    \mm$, $\mm^{00}$ is the second-largest prime ideal in $R$, and
    $\mm^{00} = P_a$ for any $a \in \mm$ such that $R/aR$ is finite.
  \end{enumerate}
\end{corollary}

\begin{lemma}\label{L:restriction is non-trivial in A}
  Let $R_0$ be an exceptional dp-minimal integral domain, and let $K_0 = \Frac(R_0)$.  If $K_0$ is ACVF-like, then there is an externally definable non-trivial valuation ring in the structure $(K_0,R_0,+,\cdot)$.
\end{lemma}
\begin{proof}
  Let $R \succ R_0$ be a sufficiently saturated elementary extension, and
  let $K = \Frac(R)$.  Then $R$ is exceptional and $K$ is ACVF-like.
  Let $\Oo = R^{00}:R^{00}$ and $\Oo_0 = \Oo \cap K$ (see Lemma \ref{lem-O-vee-def}).  Then $\Oo_0$ is
  an externally definable valuation ring on $K_0$.  We only need to
  show that $\Oo_0$ is non-trivial.

  Let $\mm_0$ and $\mm$ be the maximal ideals of $R_0$ and $R$.  Take
  non-zero $a \in \mm_0 \subseteq \mm$.  By
  Corollary~\ref{cor-two-cases}, $\mm = \sqrt{\mm^{00}}$, and so $a^n
  \in \mm^{00} = R^{00}$ for some $n$.  On the other hand, $1 \notin
  \mm^{00}$.  Since $\mm^{00}$ is an ideal in $\Oo$, we must have
  $v(a^n) \ne v(1)$, where $v$ is the valuation on $K$ induced by
  $\Oo$.  Then the restriction of $v$ to $K_0$ is non-trivial, and
  $\Oo_0$ is non-trivial.
\end{proof}

By \cite[Proposition 3.5]{HaHaJaVtop}, every sufficiently saturated field with a definable V-topology admits a non-trivial externally definable valuation ring. The following gives the same conclusion by assuming the existence of a definable infinite proper additive subgroup, rather than assuming saturation.

\begin{proposition}\label{P:ext-val-2}
    Let $(K,+,\cdot,A)$ be a dp-minimal expansion of a field $K$ by an infinite, proper additive subgroup $(A,+)
  \subsetneq (K,+)$.  Then there is a non-trivial externally definable
  valuation ring on $K$.
\end{proposition}
\begin{proof}
    By Proposition \ref{P:A-bdd}, $R=A:A$ is an infinite definable proper subring of $K$ and $\Frac (A)=K$. If it is a valuation ring, there is nothing to show so we assume it is exceptional.
    
    If $\Frac(R)$ is $p$CF-like, then it admits a definable valuation (the canonical valuation). So we assume $\Frac(R)$ is ACVF-like. The result now follows from Lemma \ref{L:restriction is non-trivial in A}.
\end{proof}
After completing the classification, we will see that ``externally
definable'' can be changed to ``definable''
(Proposition~\ref{P:ext-val-3}).

\subsection{The $p$CF-like case} \label{ss:pcf-like}

Let $R$ be a dp-minimal domain with maximal ideal $\mm$, not necessarily pure, and $K=\Frac (R)$.  In the following, we will frequently use the fact that if $R$ is exceptional then $R\subseteq \Oo_0$ for any externally definable valuation $\Oo_0$ of $K=\Frac (R)$ (Remark \ref{R:r-vs-o-exceptional}).

\begin{proposition}\label{P:construction holds}
    Let $R$ be an exceptional dp-minimal domain such that $K =
    \Frac(R)$ is $p$CF-like.  Let $\Oo_{\can}$ be the canonical
    henselian valuation on $K$ as in Remark~\ref{ocan-00}.
    \begin{enumerate}
    \item $R$ has finite index in $\Oo_{\can}$.
    \item In fact, there is some $n$ such that $\Oo_{\can} \supseteq R
      \supseteq p^n \Oo_{\can}$, and each inclusion has finite index.
    \end{enumerate}

In particular, $R$ arises via the construction in
Proposition~\ref{prop-construction} in a definable manner: $\Oo_{\can}$ is a definable valuation subring of $K$ and $R$ is the preimage of $R/p^n\Oo_{\can}$ under the definable map $\Oo_{\can}\to\Oo_{\can}/p^n\Oo_{\can}$. 
\end{proposition}
\begin{proof}
  Replacing $R$ with an elementary extension, we may assume $R$ is
  sufficiently saturated for both points.
  
   \textit{1.} By Remark~\ref{R:r-vs-o-exceptional}, $R \subseteq \Oo_{\can}$.  As both rings are definable, it
    suffices to show $\Oo_{\can}/R$ is finite, or equivalently, small.
    Equivalently, we must show that that $R^{00} = \Oo_{\can}^{00}$.
    Suppose not.  Then $R^{00} \subsetneq \Oo_{\can}^{00} \subseteq
    \mm_{\can}$, where $\mm_{\can}$ is the maximal ideal of $\Oo_{\can}$.

    Let $\pp = R^{00} = \mm^{00}$.  By Corollary~\ref{cor-two-cases},
    $\pp$ is the second-largest prime ideal of $R$.  By
    \cite[Proposition 3.8, Theorem 3.9]{dp-min-domains}, the
    localization $R_{\pp}$ is a valuation ring with maximal ideal
    $\pp$.  The fact that $\pp \subseteq \mm_{\can}$ implies the
    reverse inclusion on valuation rings: $\Oo_{\can} \subseteq
    R_{\pp}$.  Then $\pp$ is closed under multiplication by elements
    of $\Oo_{\can}$, so $\pp$ is an ideal in $\Oo_{\can}$.
    
    Let $v$ be the valuation on $K$ associated to $\Oo_{\can}$.  Both
    $\pp$ and $\Oo_{\can}^{00}$ are ideals in $\Oo_{\can}$, so they are
    defined by cuts in the value group.  The fact that $\pp \subsetneq
    \Oo_{\can}^{00}$ implies that there is some closed ball $B_{\ge\gamma}(0)$ (with respect to the valuation induced by $\Oo_{\can}$) 
    separating the two:
    \begin{equation*}
      \pp \subseteq B_{\ge\gamma}(0) \subseteq \Oo_{\can}^{00}.
    \end{equation*}
    By Fact~\ref{profinite-r/r00}, the compact Hausdorff ring
    $R/R^{00}$ is profinite, which means that $R^{00}$ is a filtered
    intersection of definable ideals $I \lhd R$ with finite index.
    The ball $B_{\ge \gamma}(0)$ is definable, so by saturation there is
    some finite-index definable ideal $I$ such that $I \subseteq B_{\ge
      \gamma}(0) \subseteq \Oo_{\can}^{00}$.  Then there is a ring
    homomorphism
    \begin{equation*}
      R/I \to \Oo/\Oo_{\can}^{00}.
    \end{equation*}
    However, $\Oo/\Oo_{\can}^{00}$ has characteristic 0
    (Remark~\ref{ocan-00}) and $R/I$ is finite, so this is absurd.

  \textit{2.} If $m$ is the index of $R$ in $\Oo_{\can}$, then $m$
    annihilates the group $(\Oo_{\can}/R,+)$, meaning that $R
    \supseteq m\Oo_{\can}$.  Write $m$ as $m_0 p^n$ where $m_0$ is
    prime to $p$.  Then $m_0$ is invertible in $\Oo_{\can}$, so
    \begin{equation*}
      \Oo_{\can} \supseteq R \supseteq m\Oo_{\can} = p^n\Oo_{\can}.
    \end{equation*}
    Finally, $\Oo_{\can}/p^n\Oo_{\can}$ has finite index because
    $\Oo_{\can}$ has finite residue field and $v(p)$ is a finite
    multiple of the minimum positive element in the valuation group
    (Fact~\ref{dpmvr}). \qedhere
\end{proof}

\subsection{Rogue domains} \label{ss:rogue}

Let $R$ be a dp-minimal domain with maximal ideal $\mm$, not necessarily pure, and $K=\Frac (R)$. Whenever $R$ is sufficiently saturated we let $\Oo$ be the externally definable valuation ring $R^{00}:R^{00}$,   $\mm_0$ its maximal ideal, $v$ be the corresponding valuation, $\Gamma$ its value group and $\mathbf{k}_0$ its residue field.

\begin{definition} \label{def-rogue}
  A \emph{rogue domain} is an exceptional dp-minimal domain $(R,+,\cdot,\dots)$, possibly with extra structure, such
  that $\Frac(R)$ is ACVF-like and $R^{00}$ is \emph{not} definable.
\end{definition}
We will
eventually show that rogue domains do not exist.
\begin{remark} \label{expand-rogue}
Remark~\ref{g00-facts} has the following consequences.
\begin{enumerate}
    \item If $R$ is a rogue domain and $R \equiv S$, then $S$ is a rogue domain.
In particular, elementary extensions of rogue domains are rogue.
\item Any
dp-minimal expansion of a rogue domain is a rogue domain.  In
particular, Shelah expansions of rogue domains are rogue.
\item If $R$ is a sufficiently saturated rogue domain, then $R/R^{00}$ is infinite.
\end{enumerate}
\end{remark}
\begin{lemma} \label{strictly-finer}
  Let $R$ be a sufficiently saturated rogue domain with fraction field $K$,
  and let $\Oo = R^{00} : R^{00}$.  Then $\Oo \subsetneq \Oo'$ for any
  definable valuation ring $\Oo'$ on $K$.
\end{lemma}
\begin{proof}
  Otherwise, $\Oo \supseteq \Oo'$, because any two externally
  definable valuation rings are comparable (essentially by Fact~\ref{F:dp-min R comparable to O}).  The definable valuation
  ring $\Oo'$ is henselian.  If its residue field is finite, then $K$
  is $p$CF-like, contradicting the fact that $K$ is ACVF-like.  Thus
  $\Oo'$ has infinite residue field.  By saturation, $\Oo'$ has large
  residue field, as does the coarsening $\Oo$.  In particular, the
  residue field $\mathbf{k}_0$ of $\Oo$ is much larger than the size
  of $R/R^{00}$.

  Let $S\subseteq R$ be a set of coset representatives of $R^{00}$ in $R$, chosen that $0\in S$. Let $v$ be the valuation associated with $\Oo$ and let $\res$ be the residue map.  Let $\rv : K^\times \to K^\times/(1+\mm_0)$ be the quotient map, where $\mm_0$ is the maximal ideal of $\Oo$.

     Let
  \begin{equation*}
    B = \{\res(x/y) : x,y \in S \setminus \{0\} \text{ and } v(x)
    = v(y)\} \subseteq \mathbf{k}_0.
  \end{equation*}
  
  Because $|R/R^{00}|$ is much smaller than $|\mathbf{k}_0|$, we have $B\subsetneq \mathbf{k}_0$.  Take $c \in
  \Oo$ such that $\res(c) \in \mathbf{k}_0^\times \setminus B$.  Note that $c
  \in \Oo^\times$ since $\res(c) \ne 0$.  The fact that $c \in \Oo^\times$ implies that
  $cR^{00} = R^{00}$, by definition of $\Oo = R^{00} : R^{00}$.  Then
  
    \[R^{00} = c(R^{00}) = (cR)^{00} \subseteq cR, \]
    so $R^{00} \subseteq cR \cap R.$
    We will arrive to a contradiction by showing that $R^{00} = cR \cap R$, and hence $R^{00}$ is definable.
    
    Otherwise, take $a_1  \in cR \cap R \setminus R^{00}$.  Then $b_1 := a_1/c \in R$.  If
  $b_1 \in R^{00}$, then $a_1 = cb_1 \in cR^{00} = R^{00}$, a
  contradiction.  So $b_1 \notin R^{00}$.

  By choice of $S$, there are some $a_2, b_2 \in S$ such that
  \begin{gather*}
    a_2 \equiv a_1 \pmod{R^{00}} \\
    b_2 \equiv b_1 \pmod{R^{00}}.
  \end{gather*}
  If $a_2 = 0$, then $a_1 \in R^{00}$, a contradiction.  So $a_2 \in S
  \setminus \{0\}$.  Similarly, $b_2 \in S \setminus \{0\}$ because
  $b_1 \notin R^{00}$.

  Note that $a_1 - a_2 \in R^{00}$ but $a_1 \notin R^{00}$.  As
  $R^{00}$ is an ideal in $\Oo$ (so given by a cut), we have    $v(a_1 - a_2) > v(a_1)$ and so $\rv(a_1)=\rv(a_2)$. Similarly, $v(b_1-b_2)>v(b_1)$ and so also $\rv(b_1)=\rv(b_2)$. Consequently, since $c\in \Oo^\times$,
  \[ \rv(c) = \rv(a_1/b_1) = \rv(a_1)/\rv(b_1) = \rv(a_2)/\rv(b_2) = \rv(a_2/b_2).\]
  In particular, $v(a_2/b_2) = v(c) = 0$, so $v(a_2) = v(b_2)$.  By
  definition of $B$,
  \begin{equation*}
    \res(c) = \rv(c) = \rv(a_2/b_2) = \res(a_2/b_2) \in B
  \end{equation*}
  contradicting the choice of $c$. 
\end{proof}

\begin{corollary} \label{weird-trick}
  Let $R$ be a rogue domain with fraction field $K$ and let $\Oo'$ be
  a definable valuation ring on $K$.  If $a \in K$ and $aR/(aR \cap
  R)$ is finite, then $a \in \Oo'$.
\end{corollary}
\begin{proof}
  We may replace $R$ with a sufficiently saturated elementary extension. By Lemma \ref{lem-O-vee-def}, $a\in \Oo=R^{00}:R^{00}$ and so we conclude by Lemma \ref{strictly-finer}.
\end{proof}

\begin{lemma}\label{lem-O-finest}
  Let $R$ be a sufficiently saturated rogue domain with fraction field $K$,
  and let $\Oo = R^{00} : R^{00}$.  Then $\Oo$ is the finest externally
  definable valuation ring on $K$.
\end{lemma}
\begin{proof}
  The ring $\Oo$ is externally definable by Corollary~\ref{C:R00:R00}.
  Fix an externally definable valuation ring $\Oo' \subseteq K$; it
  suffices to show that $\Oo \subseteq \Oo'$.  By
  Lemma~\ref{lem-O-vee-def}, we must show that if $a \in K$ and
  $aR/(aR \cap R)$ is finite, then $a \in \Oo'$.  This holds by
  Corollary~\ref{weird-trick} applied to the Shelah expansion $R^{Sh}$ of
  $R$, in which $\Oo'$ becomes a definable valuation ring (but
  saturation is lost).  The Shelah expansion is still a rogue domain
  (Remark~\ref{expand-rogue}).
\end{proof}

\begin{corollary}\label{C: no def additive subgroup}
  Let $R$ be a sufficiently saturated rogue domain with fraction field $K = \Frac(R)$, let $\Oo$ be the externally definable valuation ring $\Oo = R^{00} : R^{00} \subseteq K$, and let $\mathbf{k}_0$ be the residue field of $\Oo$.
  \begin{enumerate}
  \item There are no externally definable non-trivial valuation rings
    on $\mathbf{k}_0$.
  \item If $A \subseteq \mathbf{k}_0$ is an externally definable additive subgroup, then $A$ is finite or $A = \mathbf{k}_0$.
  \end{enumerate}
\end{corollary}
\begin{proof}
  The first point is a direct consequence of Lemma~\ref{lem-O-finest}: if there were a non-trivial externally definable valuation on $\mathbf{k}_0$ then we could compose it with $\Oo$, getting a new externally definable valuation $\Oo' \subsetneq \Oo$, contradicting the lemma.  The second statement is now a direct consequence of Proposition \ref{P:ext-val-2}, noting that saturation was not assumed there.
\end{proof}

To a first approximation, the following lemma says that rogue domains
cannot arise from Proposition~\ref{prop-construction}.
\begin{lemma} \label{constructed-aren't-rogue}
  Let $R$ be a rogue domain.  Let $\Oo'$ be an externally definable
  valuation ring on $K = \Frac(R)$ and let $I$ be an ideal in $\Oo'$.
  If $I \subseteq R$, then the group $R/I$ is infinite.
\end{lemma}
\begin{proof}
  Note that $I$ is externally definable.  Passing to the Shelah
  expansion, we may assume that $\Oo'$ and $I$ are definable.  Passing
  to an elementary extension, we may assume the structure is sufficiently
  saturated.  Both changes are acceptable by Remark~\ref{expand-rogue}.
  Because $K$ is ACVF-like, the residue field of $\Oo'$ must be
  infinite.  Then every definable ideal is 00-connected by
  Fact~\ref{F:joh19}(2), so $I = I^{00}$.  If $R/I$ is finite, then
  $R^{00} = I^{00} = I$, and $R^{00}$ is definable, contradicting the
  fact that $R$ is a rogue domain.
\end{proof}

Suppose the rogue domain $R$ is sufficiently saturated.  As $R^{00}$ is an ideal in $\Oo$, it is defined by a cut
$\Xi$ in the value group:
\begin{equation*}
  R^{00} = \{x \in K : v(x) > \Xi\}.
\end{equation*}
The cut $\Xi$ cannot be the cut $+ \infty$ or $-\infty$, since
$R^{00}$ is neither $0$ nor $K$.
\begin{lemma} \phantomsection \label{structure-lemma}
  \begin{enumerate}
  \item The value group $\Gamma$ is densely ordered (not necessarily
    divisible).
  \item $\Xi$ is not the cut $\gamma^+$ infinitesimally above some
    $\gamma \in \Gamma$.  Equivalently, the set $\{\gamma \in \Gamma :
    \gamma < \Xi\}$ has no maximum.
  \item There is an increasing sequence $\gamma_0 < \gamma_1 <
    \gamma_2 < \cdots$ in $\Gamma$, such that the following things hold:
    \begin{gather*}
      R/\{x \in R : v(x) \ge \gamma_i\} \text{ is finite for each $i$}
      \\ \lim_{i \to \infty} \left|R/\{x \in R : v(x) \ge
      \gamma_i\}\right| = \infty.
    \end{gather*}
  \item If $\gamma_0 < \gamma_1 < \gamma_2 < \cdots$ is a sequence as
    in the previous point, then $\Xi$ is the limit of the sequence, in
    the sense that
    \begin{equation*}
      \{x \in \Gamma : x > \Xi\} = \bigcap_{i = 0}^\infty \{x \in
      \Gamma : x > \gamma_i\}.
    \end{equation*}
  \end{enumerate}
\end{lemma}
\begin{proof}
  Each closed ball $B_{\ge \gamma}(0)=\{x\in K:v(x)\geq \gamma\}$ is $\vee$-definable, because it
  is a scaled copy of the $\vee$-definable ring $\Oo$.  Similarly,
  each open ball $B_{> \gamma}(0)$ is type-definable, being a scaled
  copy of the type-definable maximal ideal $\mm \lhd \Oo$.

  If $\Gamma$ is not densely ordered, then it is discretely ordered,
  and every closed ball is an open ball (of a different radius).  In
  particular, $\Oo$ is type-definable, and therefore definable,
  contradicting Lemma~\ref{strictly-finer} which says that $\Oo$ is
  strictly finer than any definable valuation ring.  Therefore
  $\Gamma$ is densely ordered, proving (1).
  \begin{claim}\label{C: claim on cut}
    If $\gamma > \Xi$, then $R/(R \cap B_{\ge \gamma}(0))$ is
    infinite.  If $\gamma < \Xi$, then $R/(R \cap B_{\ge \gamma}(0))$
    is finite, and $R \cap B_{\ge \gamma}(0)$ is definable.
  \end{claim}
  \begin{claimproof}
    If $\gamma > \Xi$, then $B_{\ge \gamma}(0) \subseteq R^{00}$ because $R^{00}$ is the ``ball of radius $\Xi$,'' so to speak.  Therefore $R \cap B_{\ge \gamma}(0) \subseteq R \cap R^{00} = R^{00}$, and
    \begin{equation*}
      \left|R : R \cap B_{\ge \gamma}(0)\right| \ge |R : R^{00}|.
    \end{equation*}
    Since $R^{00}$ is not definable, $|R:R^{00}|$ is not finite.
    
    Conversely, if $\gamma < \Xi$, then $B_{\ge \gamma}(0) \supseteq R^{00}$.  The intersection $R \cap B_{\ge \gamma}(0)$ is $\vee$-definable and
    \begin{equation*}
      \left|R : R \cap B_{\ge \gamma}(0)\right| \le |R : R^{00}|,
    \end{equation*}
    so $R \cap B_{\ge \gamma}(0)$ has small index in $R$.  In general,
    $\vee$-definable subgroups of small index have finite index and
    are definable, by a compactness argument (analogous to the fact
    that open subgroups of a compact Hausdorff group are clopen of
    finite index).
  \end{claimproof}
  Now we can prove (2).  Suppose for the sake of contradiction that
  $\Xi = \gamma^+$ for some $\gamma\in \Gamma$, so that
  \begin{equation*}
    R^{00} = \{x \in K : x > \gamma\} = B_{> \gamma}(0).
  \end{equation*}
  Note that
  \begin{equation*}
    R \supseteq R \cap B_{\ge \gamma}(0) \supseteq R \cap B_{>
      \gamma}(0) = B_{> \gamma}(0) = R^{00}.
  \end{equation*}
  The element $\gamma$ is less than the cut $\gamma^+$, so $R/(R \cap
  B_{\ge \gamma}(0))$ is finite by the claim.  On the other hand,
  $R/B_{>\gamma}(0) = R/R^{00}$ is infinite because $R$ is a rogue
  domain.  Therefore, $(R \cap B_{\ge \gamma}(0))/B_{>\gamma}(0)$ is
  infinite.  It is an externally definable subgroup of $B_{\ge
    \gamma}(0)/B_{> \gamma}(0)$. Since the latter is (externally) definably
  isomorphic to the residue field of $\Oo$,  by Corollary~\ref{C: no
    def additive subgroup}, we must have
  \begin{gather*}
    (R \cap B_{\ge \gamma}(0))/B_{>\gamma}(0) = B_{\ge \gamma}(0)/B_{> \gamma}(0) \\
    R \cap B_{\ge \gamma}(0) = B_{\ge \gamma}(0) \\
    R \supseteq B_{\ge \gamma}(0).
  \end{gather*}
  As $B_{\geq \gamma}(0)\supseteq R^{00}$,  $R/B_{\ge \gamma}(0)$ has finite index and $B_{\ge
    \gamma}(0)$ is an ideal in an externally definable valuation ring,
  contradicting Lemma~\ref{constructed-aren't-rogue}.  This proves
  (2).  In particular, the set $\{\gamma \in \Gamma : \gamma < \Xi\}$ has no
  maximum.

  For $\gamma \in \Gamma$, let $f(x) = \left|R : R \cap B_{\ge
    \gamma}(0)\right|$.  Clearly, $f$ is weakly increasing:
  \begin{equation*}
    \gamma \le \gamma' \implies R \cap B_{\ge \gamma}(0) \supseteq R \cap B_{\ge \gamma'}(0) \implies f(\gamma) \le f(\gamma').
  \end{equation*}
  Moreover, $f(\gamma)$ is finite if and only if $\gamma < \Xi$, by
   Claim \ref{C: claim on cut}.  To prove (3) and (4), it suffices to show that $f(\gamma)$ is
  unbounded as $\gamma$ approaches $\Xi$ from below.
  
  Suppose not.  Then there is some $\gamma_0 < \Xi$ and some $n$ such that
  \begin{equation*}
    \gamma_0 < \gamma < \Xi \implies f(\gamma) = n.
  \end{equation*}
  It follows that $R \cap B_{\ge \gamma}(0)$ must be some fixed
  subgroup $G$ of $R$ of index $n$, independent of $\gamma>\Xi$.  Moreover, $G$
  is definable by Claim \ref{C: claim on cut}.

  On the other hand, the fact that there is no maximum element below
  $\Xi$ means that
  \begin{align*}
    R^{00} &= \{x \in K : v(x) > \Xi\} \\
    &= \{x \in K : v(x) \ge \gamma \text{ for every } \gamma\in \Gamma \text{ with } \gamma< \Xi\} \\
    &= \bigcap_{\gamma < \Xi} B_{\ge \gamma}(0)
  \end{align*}
  and so
  \begin{equation*}
    R^{00} = R \cap R^{00} = \bigcap_{\gamma < \Xi} (R \cap B_{\ge
      \gamma}(0)) = G,
  \end{equation*}
  so $R^{00}$ is definable, contradicting the fact that $R$ is rogue.
\end{proof}

\begin{lemma} \label{circle-around}
  Let $R$ be a rogue domain.  Let $\Oo'$ be an externally definable
  valuation ring and let $v'$ be the corresponding valuation.  There cannot exist an
  ascending sequence $\gamma'_0 < \gamma'_1 < \gamma'_2 < \cdots$ in the
  value group of $v'$ such that
  \begin{gather*}
    R/\{x \in R : v'(x) \ge \gamma'_i\} \text{ is finite for each } i \\
    \lim_{i \to \infty} \left|R/\{x \in R : v'(x) \ge \gamma'_i\}\right| = \infty.
  \end{gather*}
\end{lemma}
\begin{proof}
    Replacing $R$ with its Shelah expansion, we reduce to the case where
  $\Oo'$ is definable.  Replacing $R$ with an elementary extension, we
  may assume that $R$ is sufficiently saturated.  Both changes preserve the
  fact that $R$ is rogue, by Remark~\ref{expand-rogue}. Furthermore, the existence of the ascending sequence is also preserved in the the elementary extension.
  
  Now let $\Oo$
  be the usual valuation ring $R^{00} : R^{00}$, and let $v :
  K \to \Gamma$ be its corresponding valuation.  By Lemma~\ref{strictly-finer}, $\Oo
  \subsetneq \Oo'$, meaning that $v'$ is a strict coarsening of $v$.
  In other words, $v'$ is the composition
  \begin{equation*}
    K \stackrel{v}{\to} \Gamma \to \Gamma/\Delta
  \end{equation*}
  for some non-trivial convex subgroup $\Delta \subseteq \Gamma$.

  Let $\gamma_i$ be an element in $\Gamma$ lifting $\gamma'_i \in
  \Gamma/\Delta$.  Note that
  \begin{equation*}
     \{x \in R : v'(x) \ge \gamma'_{i+1}\} \subseteq \{x \in R : v(x) \ge \gamma_i\} \subseteq \{x \in R : v'(x) \ge \gamma'_i\}
  \end{equation*}
  and so
  \begin{gather*}
    R/\{x \in R : v(x) \ge \gamma_i\} \text{ is finite for each } i \\
    \lim_{i \to \infty} \left|R/\{x \in R : v(x) \ge \gamma_i\}\right| = \infty.
  \end{gather*}
  By Lemma~\ref{structure-lemma}, the cut $\Xi$ defining $R^{00}$ is
  the limit of the $\gamma_i$:
  \begin{equation*}
    x \in R^{00} \iff v(x) > \Xi \iff (v(x) > \gamma_i \text{ for every } i) \iff (v(x) \ge \gamma_i \text{ for every } i).
  \end{equation*}
  For fixed $i$, we have
  \begin{equation*}
    v'(x) > \gamma'_i \implies v(x) > \gamma_i \implies v(x) \ge \gamma_i \implies v'(x) \ge \gamma'_i.
  \end{equation*}
  Therefore, the following four conditions are sequentially weaker:
  \begin{enumerate}
  \item $v'(x) > \gamma'_i$ for every $i$.
  \item $v(x) > \gamma_i$ for every $i$.
  \item $v(x) \ge \gamma_i$ for every $i$.
  \item $v'(x) \ge \gamma'_i$ for every $i$.
  \end{enumerate}
  But the first and fourth are equivalent because the sequence
  $\gamma'_0, \gamma'_1, \ldots$ is strictly increasing.  Thus all
  four conditions are equivalent to each other, and also to the
  condition $x \in R^{00}$.  In summary,
  \begin{equation*}
    x \in R^{00} \iff (v'(x) > \gamma'_i \text{ for all } i).
  \end{equation*}
  It follows that $\Oo' \cdot R^{00} \subseteq R^{00}$, or
  equivalently that $\Oo' \subseteq R^{00} : R^{00} = \Oo$, contradicting
  Lemma~\ref{strictly-finer}.
\end{proof}

Of course we can apply Lemma~\ref{circle-around} to $\Oo' = \Oo$
itself, and the conclusion directly contradicts parts (3) and (4) of
Lemma~\ref{structure-lemma}.  Therefore, \emph{rogue domains do not
exist}, and we have proven the following:
\begin{proposition} \label{last-prop}
  Let $R$ be a sufficiently saturated exceptional dp-minimal domain such that $\Frac(R)$ is
  ACVF-like.  Then $R^{00}$ is definable and has finite index in $R$.
  As a consequence, the valuation ring $\Oo = R^{00} : R^{00}$ is
  definable as well.
\end{proposition}
  With
Propositions~\ref{P:acf-pcf-like cover} and \ref{P:construction
  holds}, we can now prove the main theorem of the paper.

\begin{theorem} \phantomsection \label{main-thm}
    Let $R$ be a dp-minimal integral domain. If $R$ is not a valuation ring, then there exists a valuation subring $\Oo$ of $K=\Frac(R)$, a proper ideal $I\trianglelefteq \Oo$, and a finite subring of $R_0$ of $\Oo/I$ such that $R$ is the preimage of $R_0$ under the quotient map $\Oo\to \Oo/I$.  The data $(\Oo,I,R_0)$ can be chosen to be definable in $R$.
\end{theorem}
\begin{proof}
    Let $R$ be a dp-minimal integral domain which is not a valuation ring, i.e., it is exceptional. By Proposition~\ref{P:acf-pcf-like cover}, $\Frac(R)$ is either ACVF-like or pCF-like. If it is the latter, we conclude by Proposition~\ref{P:construction holds}, so assume that it is ACVF-like. There is no harm in assuming that $R$ is sufficiently saturated. By Proposition~\ref{last-prop}, $I = R^{00}$ is a definable ideal in the dp-minimal definable valuation ring $\Oo=R^{00}:R^{00}$, and $R^{00}$ has finite index in $R$.  The desired conclusion follows easily: take $R_0 = R/R^{00} \subseteq \Oo/R^{00}$.
\end{proof}

Incidentally, we can now strengthen Proposition~\ref{P:ext-val-2}:
\begin{proposition} \label{P:ext-val-3}
  Let $(K,+,\cdot,A)$ be a dp-minimal expansion of a field $K$ by an
  infinite, proper additive subgroup $(A,+) \subseteq (K,+)$.  Then
  there is a non-trivial definable valuation ring on $K$.
\end{proposition}
\begin{proof}
  There is no harm in assuming sufficient saturation. By Proposition~\ref{P:A-bdd}, the ring $R = A:A$ is an infinite
  definable proper subring of $K$ with $\Frac(R) = K$.  If $R$ is a
  valuation ring, we are done.  Otherwise, $R$ is exceptional.  If $K$
  is $p$CF-like, then the canonical valuation is a non-trivial
  definable valuation.  If $K$ is ACVF-like, Proposition~\ref{last-prop}
  gives a non-trivial valuation ring $\Oo = R^{00} : R^{00}$.
\end{proof}

\section{Remarks on dp-minimal commutative rings}
\label{S:ring-remarks}
Having classified dp-minimal integral domains, it is natural to ask what can be said about more general dp-minimal commutative rings.  As a first step, we prove the following:
\begin{proposition} \phantomsection \label{P:non-domains}
    \begin{enumerate}
        \item Every dp-minimal (commutative) ring has the form $R \times S$, where $R$ is a dp-minimal henselian local ring and $S$ is a finite ring.
        \item If $R$ is a dp-minimal local ring, then the prime ideals of $R$ are linearly ordered.  Consequently, there is a unique minimal prime ideal.
        \item If $R$ is a dp-minimal local ring, then every prime ideal containing the zero divisors is comparable to any principal ideal. 
    \end{enumerate}
\end{proposition}
\begin{proof}
    \begin{enumerate}
        \item If $R$ is a dp-finite commutative ring, then $R$ decomposes as a finite direct product $R_1 \times R_2 \times \cdots \times R_n$ where each $R_i$ is a henselian local ring \cite[Theorem~1.3]{johnson2023}.  The decomposition is definable, and so $\dpr(R) = \sum_{i=1}^n \dpr(R_i)$.  In the case where $R$ is dp-minimal, it follows that one of the $R_i$ is a dp-minimal henselian local ring and the rest of the factors are finite.
        \item Otherwise, there are two non-maximal prime ideals $\pp$ and $\qq$ which are incomparable.  Then $(\pp+\qq)/\pp$ is a non-zero ideal in the integral domain $R/\pp$.  Since $\pp$ is non-maximal, $R/\pp$ is not a field, so $R/\pp$ is infinite and every non-zero ideal is infinite.  Thus $(\pp + \qq)/\pp$ is infinite.  Similarly, $(\pp+\qq)/\qq$ is infinite.  Equivalently, $\pp/(\pp \cap \qq)$ and $\qq/(\pp \cap \qq)$ are infinite.  As $\pp, \qq$ are externally definable, this contradicts dp-minimality (Fact~\ref{almost-comparable}).
        \item Let $\pp$ be a prime ideal of $R$. Recall that there is a ring homomorphism $\phi: R\to R_\pp$ with \[\ker\phi = \{c\in R\mid \exists s\notin \pp\ cs =0\}.\]
        We denote by $R^\phi,\pp^\phi,a^\phi,$ etc the images in $R_\pp$. Note that $\pp^\phi$ is a prime ideal of $R^\phi$ and $\mm^\phi$ is the maximal ideal of $R^\phi$. 
        \begin{claim}
            $\pp^\phi = \pp^\phi R_{\pp}$
        \end{claim}
\begin{claimproof}
        It is easy to check that $R^\phi \cap \pp^\phi R_{\pp} = \pp^\phi$, so it is enough to prove that $R^\phi = R^\phi+\pp^\phi R_{\pp}$. We may assume that $\pp^\phi$ is non-maximal in $R^\phi$. By Facts \ref{F:prime ideals are ext} and \ref{F:K-R-rank}, $(R_{\pp},R^\phi,\pp^\phi,\ldots)$ is dp-minimal hence by Fact \ref{almost-comparable}, one of $R^\phi/\pp^\phi$ or $\pp^\phi R_{\pp}/\pp^\phi$ has finite index as an abelian group. As $\pp^\phi$ is non-maximal, $R/\pp^\phi$ is infinite hence $\pp^\phi R_{\pp}/\pp^\phi$ is finite. In particular, extending a set of representative of $\pp^\phi R_{\pp}/\pp^\phi$ by the element $1$ yields that $R^\phi+\pp^\phi R_{\pp}$ is a finitely generated $R^\phi$-module. Let $\mm^\phi$ be the maximal ideal of $R^\phi$, then one easily checks that $R^\phi+\pp^\phi R_{\pp}$ is a local ring with maximal ideal $\mm^\phi+\pp^\phi R_{\pp}$. Since $R^\phi/\mm^\phi$ surjects onto $(R^\phi+\pp^\phi R_\pp)/(\mm^\phi+\pp^\phi R_\pp)$, $1+\mm^\phi+\pp^\phi R_\pp$ generates $(R^\phi+\pp^\phi R_\pp)/(\mm^\phi+\pp^\phi R_\pp)$ as an $R^\phi$-module. Restricting the module action from $R^\phi$ to the Jacobson ideal $\mm^\phi$ of $R^\phi$, we have $\mm^\phi(R^\phi+\pp^\phi R_{\pp}) = \mm^\phi+\pp^\phi R_{\pp}$. By Nakayama's lemma, the generator $1+\mm^\phi+\pp^\phi R_\pp$ of the $R^\phi$-module $(R^\phi+\pp^\phi R_{\pp})/(\mm^\phi+\pp^\phi R_\pp)$ lifts to a generator of $R^\phi+\pp^\phi R_{\pp}$ as an $R^\phi$-module, so $R^\phi = R^\phi+\pp^\phi R_{\pp}$. 
\end{claimproof}

        To conclude, assume that $\pp$ contains all the zero divisors of $R$. Then the map $\phi:R\to R_\pp$ is injective. Let $a\in R$. If $a\in \pp$ then we are done, otherwise, $a\notin \pp$ hence $(1/a^\phi) \pp^\phi \seq \pp^\phi R_{\pp} = \pp^\phi$, so that if $b\in \pp$ there exists $c\in \pp$ such that $b^\phi = a^\phi c^\phi$. As $\phi$ is injective, $b = ac$, so we conclude that $\pp\seq aR$.
        \qedhere
    \end{enumerate}
\end{proof}

The comparability of prime ideals fails for rings of dp-rank 2: consider the fiber product $\Oo \times_k \Oo$ where $(K,\Oo) \models \ACVF$ and $k$ is the residue field.

If $R$ is a henselian local dp-minimal ring and $\pp$ is the unique minimal prime ideal, then $\pp$ is the nilradical $\sqrt{0}$, so every element of $\pp$ is nilpotent.  The quotient $R/\pp$ is a dp-minimal integral domain, whose structure we understand by the main theorems of this paper.  

\appendix

\section{Appendix: proof of Fact~\ref{dpm-fields-class}}
\label{appendix}
Recall Fact~\ref{dpm-fields-class}:
\begin{fact}\label{F:appendix}
  An infinite field $(K,+,\cdot)$ is dp-minimal if and only if there
  is a henselian defectless valuation ring $\Oo \subseteq K$ with
  maximal ideal $\mm$ such that
  \begin{enumerate}
  \item The value group $\Gamma := K^\times/\Oo^\times$ is dp-minimal as an
    ordered abelian group (possibly trivial).
  \item The residue field $k := \Oo/\mm$ is algebraically closed, real
    closed, or $p$-adically closed for some prime $p$.
  \item If the residue field $k$ is algebraically closed of
    characteristic $p > 0$, then the interval $[-v(p),v(p)] \subseteq
    \Gamma$ is $p$-divisible, where $v(p) = +\infty$ when
    $\characteristic(K) = p$.
  \end{enumerate}
\end{fact}
\begin{proof}
The ``if'' direction holds by the classification of dp-minimal
valuation rings (Fact~\ref{dpmvr} above, originally from
\cite[Theorems~1.5, 1.6]{johnson21classificationdpmin}), together with
the the well-known fact that $\ACF$, $\RCF$, and $p$CF are dp-minimal.
It remains to prove the ``only if'' direction.  Suppose $K$ is dp-minimal.
\begin{claim}
  There is a henselian valuation ring $\Oo$ on $K$ whose residue field
  is $p$-adically closed, real closed, or algebraically closed.
\end{claim}
\begin{claimproof}
  Let $\Oo_\infty$ be the intersection of all 0-definable valuation
  rings on $K$.  By \cite[Theorem~1.2]{johnson21classificationdpmin},
  $\Oo_\infty$ is a henselian valuation ring whose residue field
  $k_\infty$ is finite, real closed, or algebraically closed.  In the
  latter two cases, we can simply take $\Oo = \Oo_\infty$.  Suppose we
  are in the first case: $k_\infty$ is finite.  Since $\Oo_\infty$ is
  henselian, the expansion $(K,\Oo_\infty)$ is dp-minimal by
  \cite[Proposition~5.14]{HaHaStrdep}.  By the classification of
  dp-minimal valuation rings (Fact~\ref{dpmvr}), $\Oo_\infty$ is
  finitely ramified.  Let $v_\infty : K^\times \to \Gamma_\infty$ be
  the valuation induced by $\Oo_\infty$.  Let $\Delta$ be the convex
  hull of $\Zz \cdot v_\infty(p)$ in $\Gamma_\infty$.  Note that
  $\Delta \cong \Zz$ by finite ramification.  Coarsening by the convex
  subgroup $\Delta$, we get a factorization of the place $K \to
  k_\infty$ into a composition of two places
  \begin{equation*}
    K \stackrel{\Gamma_\infty/\Delta}{\longrightarrow} k
    \stackrel{\Delta}{\longrightarrow} k_\infty,
  \end{equation*}
  where the labels on the arrows show the value groups.  The fact that
  $K \to k_\infty$ is henselian implies that $K \to k$ and $k \to
  k_\infty$ are henselian.  The fact that $v_\infty(p) \in \Delta$
  implies that $K \to k$ is equicharacteristic 0 and $k \to k_\infty$
  is mixed characteristic.  Then $k \to k_\infty$ is a finitely
  ramified henselian valuation with value group $\Delta \cong \Zz$ and
  finite residue field, so $k$ is $p$-adically closed.  Take $\Oo$ to
  be the valuation ring associated to $K \to k$.
\end{claimproof}
If $\Oo$ is the henselian valuation ring from the claim, then $\Oo$
satisfies condition (2) of Fact~\ref{dpm-fields-class}.  Moreover,
$(K,\Oo)$ is a dp-minimal valued field by
\cite[Proposition~5.14]{HaHaStrdep}.  By the classification of
dp-minimal valued fields (Fact~\ref{dpmvr}), the valuation is
defectless and satisfies conditions (1) and (3) of
Fact~\ref{dpm-fields-class}.
\end{proof}
\bibliographystyle{alpha}
\bibliography{biblio}

\end{document}